\documentclass[11pt]{article}
\usepackage[T1]{fontenc}
\usepackage{latexsym,amssymb,amsmath,amsfonts,amsthm}
\usepackage{graphics}
\usepackage{enumerate}
\usepackage{makecell}
\usepackage{graphicx}
\usepackage{algorithmicx}
\usepackage{algorithm}
\usepackage{mathrsfs}
\usepackage{subfigure}
\usepackage{color}
\usepackage{hyperref}
\usepackage{accents}
\makeatletter
\def\widebar{\accentset{{\cc@style\underline{\mskip14mu}}}}
\makeatother
\hypersetup{hidelinks}
\newcommand{\R}{{\mathbb R}}

\newcommand{\be}{\begin{eqnarray}}
\newcommand{\ben}{\begin{eqnarray*}}
\newcommand{\en}{\end{eqnarray}}
\newcommand{\enn}{\end{eqnarray*}}

\newcommand{\pa}{\partial}

\newcommand{\g}{\gamma}

\newcommand{\wih}{\widehat}
\newcommand{\wit}{\widetilde}
\newcommand{\Calderon}{Calder\'{o}n}
\newcommand{\bds}[1]{\boldsymbol{#1}}

\newtheorem{theorem}{Theorem}[section]

\newtheorem{remark}{Remark}[section]
\newtheorem{prop}[theorem]{Proposition}

\definecolor{hw}{rgb}{1,0,0}
\definecolor{lk}{rgb}{0,0,1}
\begin{document}
\renewcommand{\theequation}{\arabic{section}.\arabic{equation}}
\title{{Learning-Enhanced Variational Regularization for Electrical Impedance Tomography via \Calderon's Method}%: Learning Regularizer from \Calderon's Method}
\thanks{The work of ZZ is supported by National Natural Science Foundation of China (Projects 12422117 and 12426312), Hong Kong Research Grants Council (15303122) and an internal grant of Hong Kong Polytechnic University (Project ID: P0038888, Work Programme: ZVX3). The work of KS is supported by the National Research Foundation of Korea(NRF) grant funded by the Korea government(MSIT)(RS-2024-00350215) }}
\date{}

\author{Kai Li\thanks{Department of Applied Mathematics, The Hong Kong Polytechnic University, Hung Hom, Hong Kong. (\texttt{kai1li@polyu.edu.hk}, \texttt{zhi.zhou@polyu.edu.hk}).}\and Kwancheol Shin\thanks{Institute of Mathematical Sciences, Ewha Womans University, 52, Ewhayeodae-gil, Seodaemun-gu, Seoul
03760, Republic of Korea (\texttt{kcshin3623@gmail.com})} \and Zhi Zhou\footnotemark[2]}

\maketitle

\begin{abstract}
This paper aims to numerically solve the two-dimensional electrical impedance tomography (EIT) with Cauchy data.
This inverse problem is highly challenging due to its severe ill-posed nature and strong nonlinearity, which necessitates appropriate regularization strategies. 
Choosing a regularization approach that effectively incorporates the \textit{a priori} information of the conductivity distribution (or its contrast) is therefore essential. 
In this work, we propose a deep learning-based method to capture the \textit{a priori} information about the shape and location of the unknown contrast using \Calderon's method. 
The learned \textit{a priori} information is then used to construct the regularization functional of the variational regularization method for solving the inverse problem. The resulting regularized variational problem for EIT reconstruction is then solved using the Gauss-Newton method. Extensive numerical experiments demonstrate that the proposed inversion algorithm achieves accurate reconstruction results, even in high-contrast cases, and exhibits strong generalization capabilities. Additionally, some stability and convergence analysis of the variational regularization method underscores the importance of incorporating \textit{a priori} information about the support of the unknown contrast.  

\vspace{.2in}

\end{abstract}

\section{Introduction}
\setcounter{equation}{0}
This paper is concerned with the inverse problem of electrical impedance tomography (EIT) in two dimensions.
This type of problem arises in various applications, such as medical imaging, geophysics, environmental sciences, and non-destructive testing (see, e.g., \cite{CMI99, BL02, AAG15}).

It is well known that the EIT problem is strongly nonlinear and severely ill-posed.
A variety of numerical reconstruction methods, including iterative and non-iterative approaches, have been proposed to recover the conductivity distribution (or its contrast) from a knowledge of the Cauchy data, with different regularization strategies.
For example, a sparse reconstruction algorithm was proposed in \cite{JBK12} to recover the conductivity distribution from the boundary measurements. 
This method adopted a Tikhonov functional that incorporates a sparsity-promoting $ l_1 $-penalty term, and an iterative algorithm of soft shrinkage type was developed to solve the EIT problem.
The work in \cite{IMR14} studied the Gauss-Newton algorithm for EIT reconstruction with various \textit{a priori} information.
A variational method with total variation regularization was proposed in \cite{HMK18} for EIT reconstruction from one boundary measurement, with rigorous stability and convergence analysis.
In \cite{AGS16}, the authors studied theoretical advantages of the linearized inverse problem of multifrequency EIT under several practical settings, where the contrast of the conductivity distribution is assumed to be sufficiently small to enable the linearization.
Then an efficient group sparse reconstruction algorithm of iterative type was proposed in \cite{AGS16} to solve the linearized EIT problem.
A Gauss-Newton-type method with inexact relaxed steps was introduced in \cite{JJK20} for EIT reconstruction.
For a comprehensive account of various inversion methods for the EIT problem, see \cite{CMI99, BL02, AAG15, IKJ15} and the references quoted therein.
It should be noted that these iterative algorithms sometimes struggle to achieve satisfactory reconstructions for conductivity distribution of high contrast.
One possible reason is that such methods may not be equipped with appropriate regularization strategies, which are typically challenging to choose in practical applications \cite{BL02, ASM19}.
See the monographs \cite{EHW96, KBN08} and the references quoted therein for a comprehensive discussion of various regularization methods of ill-posed and inverse problems.

Recently, non-iterative algorithms have attracted considerable attention in the field of EIT, as they avoid the need for solving the forward problem iteratively and are therefore computationally efficient. 
As a pioneering work, in his seminal paper \cite{CAP80}, \Calderon\ proved that the linearized EIT problem admits a unique solution under the assumption of small contrast, and also proposed a direct reconstruction method. 
We refer to the method in \cite{CAP80} as \Calderon's method for EIT. 
The applicability of \Calderon's method to experimental EIT data, such as data collected from saline-filled tanks or human subjects, was not demonstrated until \cite{BJM08}. 
\Calderon's approach employs the Fourier transform and special complex harmonic functions, now known as \textit{complex geometrical optics} solutions (CGO). 
The adoption of the CGO solutions and the discovery of the property of the completeness of the products for harmonic functions in $\R^n\; (n\geq 2)$ have yielded significant results, especially in the theory of inverse problems.
By employing CGO solutions, the global uniqueness of the EIT problem was established in \cite{SJU87} for spatial dimensions $n \geq 3$. 
For the two-dimensional case, global uniqueness was proved in \cite{NAI96}, where the proof led to the development of the D-bar method \cite{SSM00, MJL03}. 
The D-bar method is a direct reconstruction technique that solves the full nonlinear inverse problem using CGO solutions and a non-physical scattering transform, i.e., a nonlinear Fourier transform. 
For a comprehensive account of the D-bar method and its recent developments, see \cite{MJL12, MJL20}.
The connection between \Calderon's method and the D-bar method was elucidated in \cite{KKL07, SKM20}, where it was demonstrated that \Calderon's method can be interpreted as a three-step linearized version of the D-bar method. 
Implementation of \Calderon's method for three-dimensional imaging with simulated data can be found in \cite{SKA21}, and the comparison to the Gauss-Newton method for simulated human chest data was done in \cite{SKA24}. 
Reconstruction for two-dimensional anisotropic conductivity with \textit{a priori} information on the anisotropic tensor was performed in \cite{MRL20}.
The direct sampling method was introduced in \cite{CYT14} for locating the support of the contrast of the conductivity distribution. 
Although non-iterative algorithms can efficiently recover the shape and location of the unknown conductivity distribution, they often suffer from low resolution and are generally less accurate than iterative algorithms.
In this paper, we propose to combine the strengths of iterative algorithms and non-iterative algorithms via a deep learning approach, thereby enhancing the reconstruction performance.

To address the ill-posedness encountered in numerically solving the EIT problem, reconstruction schemes typically require regularization strategies to constrain the solution space and ensure stable conductivity recovery.
The effective incorporation of \textit{a priori} information enhances various inversion schemes through regularization, simultaneously mitigating ill-posedness and improving image resolution. 
For example, many effective inversion algorithms for the EIT problem have been developed, incorporating \textit{a priori} information such as the smallness of conductivity anomalies or the dependence of material parameters on frequency.  
We refer to \cite{AHK04, AHK07, AHB14, AGS16} and the references quoted therein for further details.
The idea of incorporating efficient \textit{a priori} information has been successfully applied in both iterative methods \cite{DDC94, ANJ95, VMV98, BUE98, DHB99, KJP99, FTD10, FDG12} and non-iterative methods \cite{AMM16, AMH17, AMM18, SKM20, SKM21} for solving the EIT problem.
However, selecting appropriate \textit{a priori} information remains challenging, as it requires precise characterization and explicit mathematical encoding of solution-specific prior knowledge specific to the application context of the inverse problem.
Instead of manual selection, in this work, we aim to use a deep neural network to automatically extract the \textit{a priori} information from ground truth data.

In recent years, there has been a growing trend of utilizing deep learning and convolutional neural networks (CNNs) for developing effective EIT reconstruction methods (see, e.g., \cite{HSJ18, FYY20, WZL19, GRJ21, CSJ23, SBZ23, LDW23, ZBZ25}).
For example, \cite{FYY20} introduced a deep neural network for solving the EIT problem under the small contrast assumption of the conductivity distribution. 
The deep neural network in \cite{FYY20} was carefully designed by studying the inherent low-rank structure of the EIT problem and was trained to invert the unknown conductivity distribution from the Dirichlet-to-Neumann (DtN) map.
In \cite{HSJ18, WZL19, GRJ21, CSJ23, SBZ23}, the authors first obtained the initial contrasts of the conductivity distribution by non-CNN-based methods, and then employed the well-trained CNNs for further refinement.
Such approaches are also categorized as post-processing methods \cite{ASM19}.
It is worth noting that \cite{CSJ23, SBZ23} used \Calderon's method to generate an initial guess for the unknown conductivity distribution, which is a non-iterative algorithm as we mentioned before, ensuring computational efficiency. 
In \cite{CMM21}, a deep neural network was trained to learn the mapping between scattering transforms in the D-bar method and the internal boundaries of anatomical structures of the thorax. 
The boundary information inferred by the trained network was then fused with the D-bar method to enhance image resolution.
In \cite{LDW23}, an untrained neural network approach was proposed for EIT reconstruction by leveraging the deep image prior (DIP), where the network architecture itself serves as an implicit regularizer to recover the unknown conductivity distribution without training data.
\cite{ZBZ25} proposed a two-stage learning-based method for solving the inverse problem of EIT.
The first stage used a well-designed neural network to obtain an initial guess from the boundary measurements, while the second stage unrolled the faster iterative shrinkage thresholding algorithm (FISTA) into a deep neural network with attention mechanisms for final reconstruction. For a comprehensive review of deep learning-based approaches to EIT and other inverse problems, see \cite{TDN23, ASM19, HMN23, LS22}.

In this paper, we propose a learning-enhanced variational regularization method via \Calderon's method (\textbf{LEVR-C}) that incorporates the \textit{a priori} information of the shape and location (i.e., the support) of the unknown contrast, which is learned by a deep neural network from \Calderon's method, as a regularization strategy for recovering the contrast of the conductivity distribution.
Precisely, we first train a deep neural network to retrieve the \textit{a priori} information of the support of the unknown contrast from \Calderon's method.
We subsequently develop a variational regularization method equipped with the learned \textit{a priori} information for EIT reconstruction, where the proposed variational formulation is then solved by the Gauss-Newton method.
Some stability and convergence analysis of the proposed variational regularization method are also provided in this paper.
Extensive numerical experiments demonstrate that our algorithm has a satisfactory reconstruction performance, stable convergence behavior, and good generalization ability.

This paper is organized as follows. 
Section \ref{S2} presents the mathematical formulation of the inverse problem of EIT.
Based on \Calderon's method, we develop a deep learning approach that can effectively extract the \textit{a priori} information of the shape and location (i.e., the support) of the unknown contrast in Section \ref{S3}.
In Section \ref{S4}, we propose a variational regularization method that effectively incorporates the learned \textit{a priori} information for the inverse problem of EIT.
This variational problem is then solved with the Gauss-Newton method for EIT reconstruction.
It should be noted that the stability and convergence analysis of our regularization method validates the important role played by the \textit{a priori} information of support of the unknown contrast.
Section \ref{S5} presents numerical experiments that demonstrate the effectiveness of the proposed regularization method.
Final conclusions and future directions are discussed in Section \ref{S6}.

\section{Problem formulation}\label{S2}
\setcounter{equation}{0}

This section introduces the forward and inverse problems of the EIT in the plane. Precisely, let the region under consideration be a bounded domain $\Omega\subset\R^2$, which contains a piecewise smooth conductivity distribution $ \sigma(x)\in L^\infty(\Omega) $ with $ \sigma>0 $. 
We assume that $\sigma(x) = m(x) + 1 $, where $m(x)$ is the contrast of the conductivity distribution from the background conductivity 1, and $m(x)$ is assumed to be compactly supported in $ \Omega $, i.e., $ \mathrm{supp}(m)\subset \Omega$. 
Then the potential $ u\in H^1(\Omega) $ in the EIT model satisfies the following equation \cite{BL02}
\begin{equation}\label{21}
\left\{ \begin{array}{ll}
-\nabla\cdot (\sigma\nabla u) = 0 & \text{in}\; \Omega,\\
\sigma\dfrac{\pa u}{\pa \nu} = g, & \text{on}\; \pa \Omega,
\end{array} \right.
\end{equation}
where the Neumann boundary data $ g\in H^{-1/2}(\pa \Omega) $ represents the current pattern imposed on the boundary, and $\nu$ is the unit outward normal vector to $ \pa \Omega $.
To ensure the solvability of the forward problem, we require $ \int_{\pa \Omega}gds = 0 $ (Kirchhoff's law), i.e.,
\begin{equation}\label{22}
g\in \wit{H}^{-1/2}(\pa \Omega) = \{v\in H^{-1/2}(\pa \Omega): \int_{\pa \Omega}vds = 0\},
\end{equation}
and we enforce $ \int_{\pa \Omega}uds = 0 $ to guarantee uniqueness of the solution, i.e.,
\begin{equation}\label{23}
u\in \wit{H}^1(\Omega) = \{v\in H^1(\Omega): \int_{\pa \Omega}vds = 0\}.
\end{equation}
The condition $ \int_{\pa \Omega}uds = 0 $ imposes a grounding condition for the electrical potential $ u $.
For a given conductivity distribution $ \sigma $, we denote by $ \Lambda_\sigma $ the Neumann-to-Dirichlet (NtD) map $ \Lambda_\sigma: H^{-1/2}(\pa \Omega)\to H^{1/2}(\pa \Omega) $ given by
\begin{equation}\label{24}
\Lambda_\sigma g = u|_{\pa \Omega}.
\end{equation}
For the simplicity of notation, we write the potential distribution on the boundary $\partial \Omega$ as $f:=u|_{\partial \Omega}$, which also represents the Dirichlet boundary data of $u$. Correspondingly, we denote by $ (\Lambda_\sigma)^{-1} $ the Dirichlet-to-Neumann (DtN) map. The forward problem of EIT is to compute the NtD map $ \Lambda_\sigma $ for a given conductivity distribution $ \sigma $.
Conversely, the inverse problem of EIT aims to reconstruct the unknown conductivity distribution $ \sigma $ (i.e., its contrast $ m=\sigma-1 $) given its corresponding NtD map $ \Lambda_\sigma $.
Theoretically, the conductivity distribution $\sigma$ as a positive $ L^\infty $ function can be recovered from the NtD map $ \Lambda_\sigma $ \cite{AKP06}.
However, in practical applications, we do not have full knowledge of the NtD map $ \Lambda_\sigma $ \cite{BL02}. 
Instead, we measure voltage $ f $ (Dirichlet boundary data) induced by applied current pattern $g$ (Neumann boundary data), yielding a set of the Cauchy data pairs $ \{(g_q, f_q)\}_{q=1}^Q $ for the systems \eqref{21}--\eqref{23}.
This paper considers the following inverse problem.

\textbf{EIT problem}: Reconstruct the conductivity contrast $m$ from measured Cauchy data pairs $ \{(g_q, f_q)\}_{q=1}^Q $.

To mathematically formalize the EIT problem, we introduce the forward operator of \eqref{21} by $\mathcal{F}(\cdot; g):
L^\infty(\Omega)\to H^1(\Omega)$, where $ g $ is the Neumann data in \eqref{21}. 
This operator maps the contrast $ m $ to its corresponding potential $ u $, that is $ \mathcal{F}(m; g) = u $.
With the trace operator be $\gamma: H^1(\Omega)\to H^{1/2}(\pa\Omega) $, the EIT problem consists in solving the following equations:
\begin{equation}\label{25}
\gamma\mathcal{F}(m; g_q) = f_q,\quad q=1,2,\cdots,Q.
\end{equation}
Note that solving these equations is nonlinear and severely ill-posed \cite{BL02}.
In practical applications, we can only acquire noisy measured data. 
Hence, we consider the noised perturbation $ f^\delta_q $ of the Dirichlet boundary data $ f_q $, where $\delta>0$ denotes the noise level (see Subsection \ref{S511} for the choice of $ f^\delta_q $). 
Thus, given the noisy Cauchy data $ \{(g_q, f^\delta_q)\}_{q=1}^Q $, we rewrite \eqref{25} as
\begin{equation}\label{26}
\gamma\mathcal{F}(m; g_q) \approx f^\delta_q,\quad q=1,2,\cdots,Q
\end{equation}
for the unknown contrast $m$.

In this paper, Newton type methods are used in our numerical algorithm as backbone.
Hence, it is necessary to compute the Fr\'{e}chet derivative of $ \mathcal{F} $.
It can be shown that $ \mathcal{F}(\cdot; g): m\mapsto u $ is Fr\'{e}chet differentiable, and the derivative is given by $ \mathcal{F}'(m;g)(q) = v $, where $ q\in L^2(\Omega) $ and $ v $ satisfies the following equation \cite{JBK12}
\begin{equation}\label{27}
\left\{ \begin{array}{ll}
-\nabla\cdot (\sigma\nabla v) = \nabla\cdot(q\nabla u) & \text{in}\; \Omega,\\
\sigma\dfrac{\pa v}{\pa \nu} = 0, & \text{on}\; \pa \Omega,
\end{array} \right.
\end{equation}
Here, $ u $ is the electrical potential corresponding to the contrast $ m $, with the electrical current $ g $ applied.
From the above argument, it follows that computing the numerical solution of the equation \eqref{27} is necessary to evaluate the Fr\'{e}chet derivative of $\mathcal{F}$.

To enable numerical reconstruction, it is necessary to discretize the contrast $ m $. 
Without loss of generality, we consider $ \Omega:=[-1,1]\times [-1,1]\in \R^2 $, which is first partitioned into a uniform $ N\times N $ grid of smaller squares (i.e., pixels), each with side length $ h=2/N $.
Let $ x_{ij}\;(i,j = 1,\cdots,N) $ denote pixel centers.
Each square is then divided into two right-angled triangles by its diagonal to form a shape regular quasi-uniform triangulation $ \mathcal{T}_h $ on $\Omega$ with mesh size $ h $ (see Figure \ref{F1} (b) for illustration).
Over $\mathcal{T}_h$, we define a continuous piecewise linear finite element space
\begin{equation}\label{28}
M_h = \{m_h\in L^\infty(\Omega): m_h|_T \text{ is a linear function } \forall T\in \mathcal{T}_h\},
\end{equation}
which will be employed to approximate the contrast $ m $.
Let $ \mathcal{I}_h $ be the Lagrange interpolation operator associated with the finite element space $ M_h $ \cite{EAG04}.
Then the contrast $ m $ can be discretized by $ \bds{m} := \mathcal{I}_h(m) $ and is also called the discrete contrast in this paper.
Similarly, the conductivity distribution $\sigma$ can be discretized by $ \bds{\sigma} := \mathcal{I}_h(\sigma) $ (see Figure \ref{F1} for an exemplary mesh).
We further introduce a discrete matrix $ \bds{S} = (\bds{S}_{ij})\in\R^{N\times N} $ with
\begin{equation}\label{211}
\bds{S}_{ij} :=
\left\{ \begin{array}{ll}
1, &  m(x_{ij})\neq 0,\\
0, &  m(x_{ij}) = 0,
\end{array} \right.
\end{equation}
to characterize $ \mathrm{supp}(m) $.
Note that the index set $ \{(i,j): \bds{S}_{ij}\} $ can be viewed as the discrete form of $ \mathrm{supp}(m) $.
We will also call $ \bds{S} $ the support matrix in this paper. 
In practical applications, we have only a finite number of $ P $ (even integer) electrodes uniformly attached on the boundary $ \pa \Omega $, which allow us to collect partial boundary data (see Figure \ref{F1} (b) for an exemplary illustration).
Let $ \{e_p\}_{p=1}^P\subset\R^2 $ denote the coordinates of the electrodes.
In our EIT simulation, we impose $ Q $ linearly independent current patterns $ \bds{g} :=(\bds{g}_q)\in\R^{P\times Q} $ (with $ \bds{g}_q $ be its column vector) on the electrodes, and then measure the corresponding noisy voltage distributions $ \bds{f}^\delta := (\bds{f}_q^\delta)\in\R^{P\times Q} $ (with $ \bds{f}_q^\delta $ be its column vector) on the electrodes.
Note that $ \bds{g}_q $ and $ \bds{f}_q^\delta $ can be viewed as approximations of the Neumann boundary data $ g_q $ and the Dirichlet boundary data $ f_q $ as we mentioned before.
To be more specific, given the $ q $-th current pattern, let $ \bds{g}_{q,p} $ and $ \bds{f}_{q,p}^\delta$ ($ p = 1,2,\cdots,P $) denote respectively the applied current and measured voltage on the $ p $-th electrodes.
Recall the conditions \eqref{22} and \eqref{23}, we require $ \sum_{p=1}^{P}\bds{g}_{q,p} = 0 $ and we adjust $\bds{f}_{q,p}^\delta$ so that $ \sum_{p=1}^{P}\bds{f}_{q,p}^\delta = 0 $.
In this work, we apply trigonometric current patterns defined as
\begin{equation}\label{29}
\bds{g}_{q,p} = g_q(e_p):=
\left\{ \begin{array}{ll}
\cos [(q+1)\theta_p/2], &  \text{for}\;1\leq q\leq Q\;\text{with $ q $ odd},\\
\sin [q\theta_p/2], & \text{for}\;1\leq q\leq Q\;\text{with $ q $ even},
\end{array} \right.
\end{equation}
where $ \theta_p := 2\pi (p-1)/P $.
Then the formula \eqref{26} can be approximated as follows:
\begin{equation}\label{210}
\bds{\gamma}\mathcal{F}^A(\bds{m}; \bds{g}_q) \approx \bds{f}_q^\delta,\quad q=1,2,\cdots,Q.
\end{equation}
where $ \mathcal{F}^A $ and $ \bds{\gamma} $ denote the discrete forms of $ \mathcal{F} $ and the trace operator $\gamma$, respectively. 
For the convenience of later use, we rewrite \eqref{210} as follows:
\begin{equation}\label{212}
\bds{F}(\bds{m}; \bds{g})\approx \bds{f}^\delta,
\end{equation}
where $ \bds{F}(\bds{m}; \bds{g}) := \left(\bds{\gamma}\mathcal{F}^A(\bds{m};\bds{g}_q)\right)\in\R^{P\times Q} $ with $ \bds{\gamma}\mathcal{F}^A(\bds{m};\bds{g}_q)\; (q=1,2,\cdots,Q) $ be its column vector. 
We denote the Fr\'{e}chet derivative of $ \bds{F} $ by $ \bds{F}' $.

\begin{figure}[htbp]
\centering
\includegraphics[width=0.8\textwidth]{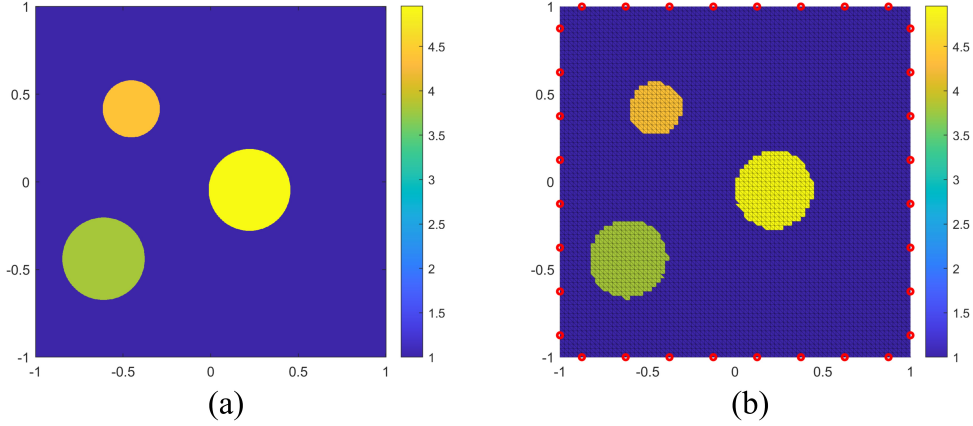}
\caption{EIT simulation for the conductivity distribution $\sigma$: (a) Original conductivity distribution $\sigma$; (b) Discretized conductivity distribution $ \bds{\sigma}=\mathcal{I}_h(\sigma) $, with $ Q=32 $ red dots indicating electrode positions.}\label{F1}
\end{figure}

\section{Learning support information from \Calderon's method}\label{S3}
\setcounter{equation}{0}
In this section, we propose a deep neural network to approximate the support of the unknown contrast, using the initial reconstruction provided by \Calderon's method.
The approximate support of the unknown contrast will then be used in the next section to develop a learning-enhanced variational regularization method for the EIT problem.

\subsection{\Calderon's method}\label{S31}

\Calderon's method is a computationally efficient, linearized, and non-iterative algorithm for EIT.
Building on the foundational work of \cite{CAP80}, it provides an approximation of the contrast $ m $ under the small-perturbation assumption that $ \|m\|_{L^\infty(\Omega)} $ is small (i.e., $\sigma\approx 1 $) with a knowledge of the DtN map $ (\Lambda_\sigma)^{-1} $.
To derive \Calderon's method, we first consider the following \textit{complex geometrical optics} solutions (CGO):
\begin{equation}\label{41}
\phi_1(x;k) = \exp(\pi ik\cdot x + \pi k^\perp\cdot x),\quad \text{and}\quad\phi_2(x;k) = \exp(\pi ik\cdot x - \pi k^\perp\cdot x),
\end{equation}
where $ k = (k_1, k_2) $ is a real-valued vector and $ k^\perp = (-k_2, k_1) $ is orthogonal to $k$. The complex-valued vectors $\zeta = k - i k^\perp$ and $\bar{\zeta} = k + i k^\perp$ serve as non-physical frequency variables in the CGO solutions, and $ x = (x_1, x_2) $ denotes the spatial variable for the domain $ \Omega $.
Let $\Lambda_1$ denote the NtD map for $ \sigma\equiv 1 $, which can be generally calculated by simulation.
Then the following approximation holds by using the inverse Fourier transform \cite{CAP80, CSJ23}:
\begin{equation}\label{42}
m(x)\approx C(x;\Lambda_\sigma) := \int_{|k|<R}H(k;\Lambda_\sigma)e^{-2\pi ik\cdot x}dk,
\end{equation}
where $ R $ is a constant controlling the frequency cutoff and
\begin{equation}\label{43}
H(k;\Lambda_\sigma) := -\dfrac{1}{2\pi^2|k|^2}\int_{\pa \Omega}\phi_1(\xi;k)\left[(\Lambda_\sigma)^{-1} - (\Lambda_1)^{-1}\right]\phi_2(\xi;k) ds(\xi).
\end{equation}
Note that truncating \( H(k;\Lambda_\sigma) \) to the region \( \{k : |k| < R\} \) results in a loss of information for high-frequency Fourier terms, and hence the obtained reconstruction tends to be blurry. This truncation can be viewed as a regularization strategy for \Calderon's method, since \( H(k; \Lambda_{\sigma}) \) approximates the Fourier transform of \( m(x) \) reasonably well only for small values of \( |k| \). This is because \( H(k; \Lambda_{\sigma}) \) is a linear approximation of the scattering transform \cite{SKM20}, and also because the high-frequency components of the data are particularly sensitive to noise: small measurement errors or fluctuations on the boundary can lead to disproportionately large errors in the reconstruction of high-frequency Fourier modes due to the ill-posedness of the inverse problem and the amplification effects inherent in the inversion process.

As mentioned in Section \ref{S2}, only the noisy Cauchy data $ (\bds{g}, \bds{f}^\delta) $ is available, where $ \bds{g} = (\bds{g}_q)\in\R^{P\times Q} $ and $ \bds{f}^\delta = (\bds{f}_q^\delta)\in\R^{P\times Q} $.
In order to visualize the approximation $ C(x;\Lambda_\sigma) $ given by \eqref{42}, we first discretize the CGO solution $ \phi_1(\xi;k) $ and $ \phi_2(\xi;k) $ as column vectors $ \bds{\phi}_{1,k} := \left[\phi_1(e_p;k)\right]\in\R^P $ and $ \bds{\phi}_{2,k} := \left[\phi_2(e_p;k)\right]\in\R^P $, respectively, where $ e_p $ ($ p=1,2,\cdots,P $) are coordinates of the electrodes (see Section \ref{S2}).
Let $ \bds{g}^\dagger $ and $ \bds{f}^\dagger $ be the Moore-Penrose pseudoinverses of the matrices $ \bds{g} $ and $ \bds{f}^\delta $, respectively.
The CGO solutions admit the following expansion:
\begin{equation}\label{44}
\bds{\phi}_{1,k} \approx \sum_{q=1}^{Q}a_{q,k}\bds{g}_q\quad\text{and}\quad \bds{\phi}_{2,k} \approx \sum_{q=1}^{Q}b_{q,k}\bds{f}_q^\delta,
\end{equation}
with coefficients $ \bds{a}_k = (a_{1,k}, a_{2,k}, \cdots, a_{Q,k})^T := \bds{g}^\dagger\bds{\phi}_{1,k} $ and $ \bds{b}_k = (b_{1,k}, b_{2,k}, \cdots, b_{Q,k})^T := \bds{f}^\dagger\bds{\phi}_{2,k} $.
Here, the subscript $ k $ indicates the dependence of the coefficients on the variable $ k $ and the superscript $ T $ refers to the transpose operator.
We now have the following approximation:
\begin{equation}\label{45}
H(k;\Lambda_{\sigma})\approx H^A\left(k; (\bds{g}, \bds{f}^\delta)\right) := -\dfrac{|\pa\Omega|}{2\pi^2|k|^2P}\bds{a}_k^T\bds{G}\bds{b}_k - \int_\Omega\exp(2\pi ik\cdot \xi)d\xi,
\end{equation}
where $ |\pa\Omega| $ denotes the length of the boundary $ \pa\Omega $, and $ \bds{G} := (G_{ij})\in\R^{Q\times Q} $ is a Gram matrix with entries $ G_{ij} := \bds{g}_i\cdot\bds{g}_j= \sum_{p=1}^{P}\bds{g}_{i,p}\bds{g}_{j,p} $.
The integral $ \int_\Omega\exp(2\pi ik\cdot \xi)d\xi $ in \eqref{45} can be computed numerically by Simpson's quadrature rule.
Then we can approximate the unknown contrast $ m(x) $ by
\begin{equation}\label{46}
m(x)\approx C(x;\Lambda_\sigma)\approx C^A\left(x;(\bds{g}, \bds{f}^\delta)\right) := \int_{|k|<R}H^A\left(k; (\bds{g}, \bds{f}^\delta)\right)e^{-2\pi ik\cdot x}dk.
\end{equation}
This integral is also evaluated via Simpson's quadrature rule.
Following Section \ref{S2}, we discretize $ \Omega = [-1,1]\times [-1,1] $ into $ N\times N $ uniformly distributed pixels $ x_{ij} $, $ i,j=1,\cdots,N $.
Then $ C^A $ can be approximately represented by a discrete matrix $ \bds{C} = (\bds{C}_{ij})\in \R^{N\times N} $ with $ \bds{C}_{ij} $ to be the real part of $ C^A\left(x_{ij};\{(\bds{g}_p, \bds{f}_p^\delta)\}_{p=1}^P\right) $ and is also called the \Calderon's matrix of the contrast $ m $ in the rest of the paper.
Here, the real part is taken because $ m $ is assumed to be real-valued in this paper (see Section \ref{S2}).
For more details of the numerical implementation of \Calderon's method, see \cite{CSJ23, BJ09} and references therein.

\begin{remark}\label{R1}\rm
\Calderon's method requires the contrast $m$ to be sufficiently small.
However, it is observed in the numerical results of \cite[Chapter 2.3]{BJ09} that \Calderon's method is still effective for recovering the shape and location (i.e., the support) of the unknown contrast even in the case of high contrast.
This motivates training a deep neural network to extract support information contained in \Calderon's method and thus provide reliable approximate support for designing a regularization functional for the EIT problem (see Section \ref{S4}).
Moreover, since reconstructions by \Calderon's method are typically blurry, as discussed in this subsection, applying a deep neural network is expected to enhance the resolution.
\end{remark}

\subsection{Deep neural network for generating approximate support}\label{S32}

In this subsection, we propose a deep neural network $\mathcal{M}_\Theta$ to extract the \textit{a priori} information of the support $ \mathrm{supp}(m) $ of the unknown contrast $ m $ (i.e., the support matrix $ \bds{S} $ defined in Section \ref{S2}) from the \Calderon's method (i.e., the \Calderon's matrix $ \bds{C} $ defined in Subsection \ref{S31}). 
Here, $\Theta$ denotes the trainable network parameters.
Once it is well trained, $\mathcal{M}_\Theta$ should provide a good approximate support matrix for the unknown contrast, which serves as important \textit{a priori} information for EIT reconstruction (see Sections \ref{S4} and \ref{S5}). 
In the following, we will introduce the architecture of the deep neural network $\mathcal{M}_\Theta$, the training strategy for finding a suitable $\mathcal{M}_\Theta$, and the method for applying $\mathcal{M}_\Theta$ to retrieve a good approximate support matrix.

\subsubsection{Network architecture}

We implement $\mathcal{M}_\Theta$ using a U-Net architecture. 
The original version of U-Net was first introduced in \cite{ROF15} for biomedical image segmentation. 
In this work, we use a modified version of U-Net (see Figure \ref{F2}), following the architecture design in \cite{LZZ24}.
The input and output of $\mathcal{M}_\Theta$ are volumes with size $ (N\times N \times 1) $ with $ N = 80 $ being the same as in Section \ref{S2}. 
See \cite{LZZ24, ROF15} for more details of the U-Net architecture.

\begin{figure}[htbp]
\centering
\includegraphics[width=.85\textwidth]{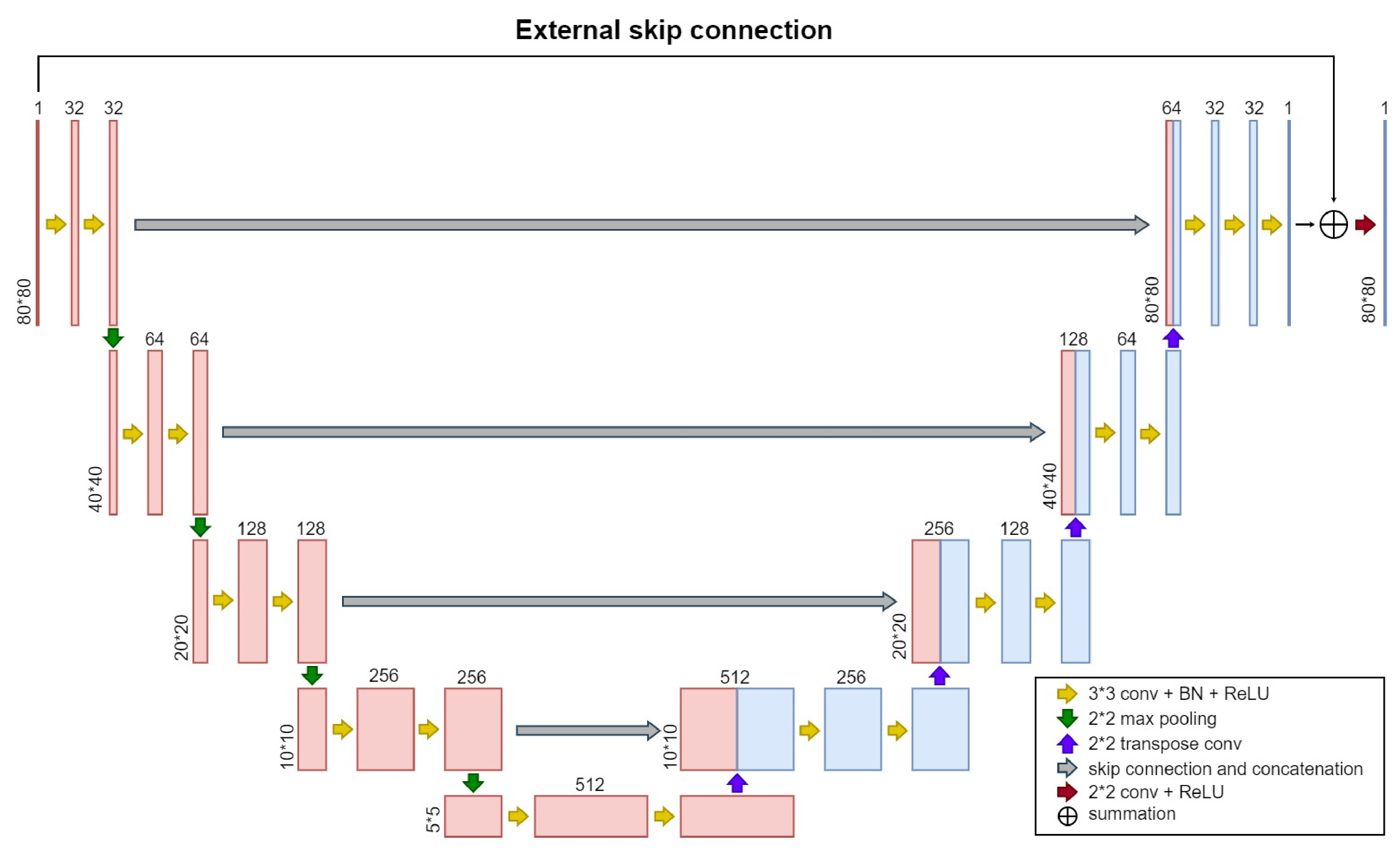}
\caption{The architecture of $\mathcal{M}_\Theta$. Each red and blue item represents a volume (also called multichannel feature map). The number of channels is shown at the top of the volume, and the length and width are provided at the lower-left edge of the volume. The arrows denote the different operations, which are explained at the lower-right corner of the figure.} \label{F2}
\end{figure}

\subsubsection{Training strategy and application}\label{S322}

The network $\mathcal{M}_\Theta$ aims to learn the mapping from normalized \Calderon's matrix $ \mathcal{N}(\bds{C}) $ to its corresponding support matrix $ \bds{S} $, where $ \bds{C} $ and $ \bds{S} $ denote the \Calderon's matrix and the support matrix given as in Subsection \ref{S31} and Section \ref{S2}, respectively. 
Here, $\mathcal{N}$ is a normalization operator defined by
$\mathcal{N}(\bds{M}):=\bds{M}/\|\bds{M}\|_{\max}$ for any
$\bds{M} := (\bds{M}_{ij})\in\R^{N\times N}$ with $\|\bds{M}\|_{\max} := \max_{1\leq i,j\leq N}|\bds{M}_{ij}| $.
As used in \cite{LZZ24}, we hope that applying $\mathcal{N}$ to a \Calderon's matrix $ \bds{C} $ can make $ \mathcal{M}_\Theta $ focus on extracting the \textit{a priori} information of the shape and location (i.e., the support) of the unknown contrast contained in its corresponding \Calderon's matrix $ \bds{C} $, and thus lead to a good approximation of the corresponding support matrices $ \bds{S} $.
During the training phase, we first generate a sample set of exact contrasts $ \{m_i\}_{i = 1}^T$ with $ T\in \mathbb{N}^+ $. 
Then we simulate the noisy Cauchy data pairs of $\{m_i\}_{i = 1}^T$, from which the corresponding \Calderon's matrices $ \{\bds{C}_i\}_{i = 1}^T $ are computed using \Calderon's method (see Subsection \ref{S31}). 
We further compute the support matrices $\{\bds{S}_i\}_{i = 1}^T$ of the sample set $\{m_i\}_{i = 1}^T$. 
The network $\mathcal{M}_\Theta$ is then trained on the following dataset with Xavier initialization \cite{GXB10}:
\begin{equation}\label{47}
D := \{(\mathcal{N}(\bds{C}_i), \bds{S}_i)\}_{i = 1}^T,
\end{equation}
where $\mathcal{M}_\Theta$ is trained to map $ \mathcal{N}(\bds{C}_i) $ to $ \bds{S}_i $, $i = 1,2,\ldots,T$, and the loss function is given as
\begin{equation}\label{48}
\mathcal{L}(\Theta) := \sum_{i=1}^T\left\|\mathcal{M}_\Theta(\mathcal{N}(\bds{C}_i)) - \bds{S}_i\right\|^2,
\end{equation}
where $ \|\cdot\| $ denotes the Frobenius norm of a matrix.
After training, we obtain the trained network $\mathcal{M}_{\widehat{\Theta}}$.

We now explain how to use our well-trained $\mathcal{M}_{\widehat{\Theta}}$.
For an unknown contrast with support matrix $\bds{S}$, we first measure its Cauchy data $ (\bds{g}, \bds{f}^\delta) $ (see Section \ref{S2}) and then compute its corresponding \Calderon's matrix $\bds{C}$ (see Subsection \ref{S31}). 
Further, we compute $ \bds{S}_A := \mathcal{M}_{\widehat{\Theta}}(\mathcal{N}(\bds{C})) $ to approximate the exact support matrix $\bds{S}$.
However, $ \bds{S}_A $ is not directly usable in practice since $ \bds{S}_A $ is not necessarily a binary matrix, which better characterizes the support area.
To address this issue, we introduce a thresholding operator $\mathcal{S}_\gamma$:
\begin{equation}\label{49}
\mathcal{S}_\g(\bds{M}):=(\mathcal{S}_\g(\bds{M}_{ij}))\in\R^{N\times N},\;\;
\mathcal{S}_\gamma(\bds{M}_{ij}):=\left\{ \begin{array}{ll}
1, & \bds{M}_{ij}>\gamma,\\
0, & \bds{M}_{ij}\leq\gamma,
\end{array} \right.
\quad \forall\bds{M}:=(\bds{M}_{ij})\in\mathbb{R}^{N\times N}.
\end{equation}
We then use $ \widetilde{\bds{S}} := \mathcal{S}_\gamma(\mathcal{M}_{\widehat{\Theta}}(\mathcal{N}(\bds{C})))\in \R^{N\times N} $ to approximate the support matrix $\bds{S}$, where $\gamma\in (0,1)$ will be chosen carefully in numerical reconstruction process (see Subsection \ref{S512} for the choice of $\gamma$).
With the above process, we are able to generate a reasonable approximation for the support of the unknown contrast, and hopefully for the high contrast case (see Remark \ref{R1}).
Since $\Omega = [-1,1]\times [-1,1]$ is partitioned into $ (N\times N) $ pixels in Section \ref{S2}, $ \widetilde{\bds{S}} $ can be naturally viewed as a function defined on $ \Omega $ by the following definition
\begin{equation}\label{50}
\widetilde{\bds{S}}(x) := \widetilde{\bds{S}}_{ij},\; \text{for}\; x\in N_h(x_{ij}) ,\quad \forall x\in \Omega,
\end{equation}
where $ N_h(x):= [x_1-h/2, x_1+h/2]\times [x_2-h/2, x_2+h/2] $ for any $ x = (x_1, x_2)\in\Omega $.
Here, $ x_{ij} $ and $ h $ are the same as in Section \ref{S2}.
Note that $ N_h(x_{ij}) $ denotes the $ (i,j) $-th pixel of $\Omega$.
Further, we have $ \widetilde{\bds{S}}\in M_h $.
Similarly, the support matrix $ \bds{S} $ can also be viewed as a function defined on $\Omega$.

\section{Learning-enhanced variational regularization method for EIT reconstruction}\label{S4}
\setcounter{equation}{0}

In this section, we propose the \textbf{LEVR-C} method for solving the EIT problem, incorporating the \textit{a priori} information of the shape and location (i.e., the support) of the unknown contrast as a regularization term, where the support information is retrieved from \Calderon's method using our well-trained deep neural network $ \mathcal{M}_{\widehat{\Theta}} $ (see Section \ref{S3}).
We highlight that the proposed variational regularization method can be applied to general inverse problems, and the stability and convergence analysis of this variational regularization method will be given under some appropriate assumptions in Appendix \ref{A}, which validates its effectiveness.
To be more specific, we consider the variational regularization method, which is one of the most popular approaches for solving the EIT problem.
This approach reformulates the inverse problem as the following optimization problem
\begin{equation}\label{51}
\underset{\bds{m}\in M_h}{\arg\min}\left\{\dfrac{1}{2}\left\|\bds{F}(\bds{m}; \bds{g}) - \bds{f}^\delta\right\|^2 + \mathcal{R}_\alpha(\bds{m})\right\},
\end{equation}
where $\mathcal{R}_\alpha: M_h \to [0, +\infty)$ is an appropriate regularization functional which encodes certain \textit{a priori} information of the exact contrast matrix, and $ \alpha>0 $ is usually chosen to be a regularization parameter that governs the influence of the \textit{a priori} knowledge encoded by the regularization functional on the need to fit data.
For the definition of the finite element space $ M_h $, see Section \ref{S2}. 
In this section, we will design the regularization functional $\mathcal{R}_\alpha$ by using the \textit{a priori} information of the support of the unknown contrast.
Precisely, we first use the deep neural network $ \mathcal{M}_{\widehat{\Theta}} $ to extract the approximate support matrix $ \widetilde{\bds{S}}\in M_h $ from \Calderon's method (see Subsection \ref{S322} for more details), and then define the regularization functional $\mathcal{R}_\alpha$ as
\begin{equation}\label{52}
\mathcal{R}_\alpha(\bds{m}) := \dfrac{1}{2}\left(\alpha\left\|\widetilde{\bds{S}}\bds{m}\right\|^2_2 + \left\|\left(1-\widetilde{\bds{S}}\right)\bds{m}\right\|^2_2\right),\quad \forall \bds{m}\in M_h,
\end{equation}
where $ \|\cdot\|_2 $ denotes the $ L^2 $-norm, and the regularization parameter $\alpha$ will be chosen carefully in numerical reconstruction process (see Subsection \ref{S512} for the choice of $\alpha$).
We note that, as $\alpha, \delta\to 0$, the regularization functional $ \mathcal{R}_\alpha $ defined above is expected to penalize large values of an approximate discrete contrast outside $ \mathrm{supp}(\wit{\bds{S}}) $ (see Proposition \ref{P1} for theoretical guarantee) and thus can hopefully lead to good reconstruction results if $ \widetilde{\bds{S}} $ is a good approximation to the exact support matrix $ \bds{S} $ (see Remark \ref{R3}).
Furthermore, under some appropriate assumptions, if $ \mathrm{supp}(\bds{m}) = \mathrm{supp}(\bds{S}) \subset \mathrm{supp}(\wit{\bds{S}}) $, the solutions of the proposed variational method \eqref{51} converge faster than the classical Tikhonov regularization (see Theorem \ref{thm54}).
The functional interpretations of $ \bds{S} $ and $ \wit{\bds{S}} $ follow from Subsection \ref{S322}.
To solve \eqref{51} with $ \mathcal{R}_\alpha $ defined as \eqref{52}, we use the Gauss-Newton method \cite{GBM07, BA24}, which updates at each iteration as follows:
\begin{equation}\label{53}
\bds{m}^\delta_{i+1} =\bds{m}^\delta_i + \left(\wit{\bds{S}}_\alpha + \left[\bds{F}'(\bds{m}^\delta_i; \bds{g})\right]^*\bds{F}'(\bds{m}^\delta_i; \bds{g}) \right)^{-1}\left[\bds{F}'(\bds{m}^\delta_i; \bds{g})\right]^*\left(\bds{f}^\delta - \bds{F}(\bds{m}^\delta_i; \bds{g})\right),
\end{equation}
where $ \bds{m}^\delta_i $ and $ \bds{m}^\delta_{i+1} $ are approximations to the unknown contrast at the $ i $-th and $ (i+1) $-th iterations, respectively, with the superscript $\delta$ indicating the dependence on the noise level.
Here, $ \left[\bds{F}'(\bds{m}^\delta_i; \bds{g})\right]^* $ is the adjoint of $ \bds{F}'(\bds{m}^\delta_i; \bds{g}) $, and $ \wit{\bds{S}}_\alpha(\bds{m}): M_h\to M_h $ is defined as
\begin{equation}\label{54}
\wit{\bds{S}}_\alpha(\bds{m}) := \alpha\wit{\bds{S}}\bds{m} + (1-\wit{\bds{S}})\bds{m},\quad \forall \bds{m}\in M_h.
\end{equation}
Since the \textit{a priori} information encoded in our variational regularization method \eqref{51} is learned by the deep neural network $ \mathcal{M}_{\widehat{\Theta}} $, the iterative method \eqref{53} is referred to as the Learning-Enhanced Variational Regularization method via \Calderon's method, abbreviated as \textbf{LEVR-C}.

We now describe the proposed \textbf{LEVR-C} for the EIT problem.
For an unknown contrast $ m $, we first measure its noisy Cauchy data $ (\bds{g}, \bds{f}^\delta) $ to obtain its \Calderon's matrix $ \bds{C} $ (see Subsection \ref{S31}) by the \Calderon's method, and then compute the approximate support matrix $ \widetilde{\bds{S}} := \mathcal{S}_\gamma(\mathcal{M}_{\widehat{\Theta}}(\mathcal{N}(\bds{C}))) $ (see Subsection \ref{S322}) and the operator $ \wit{\bds{S}}_\alpha $ as \eqref{54}. 
In the reconstruction process, the initial guess of $ \bds{m} $ is set to $ 0 $.
Let $ N_v $ be the total iteration number.
Then we can present \textbf{LEVR-C} in Algorithm \ref{V} for the EIT problem.
See Section \ref{S5} for the performance of our \textbf{LEVR-C}.

\begin{algorithm}[htbp]
\caption{Learning-enhanced variational regularization via \Calderon's method (\textbf{LEVR-C})}\label{V}

\textbf{Input: }$\bds{F}$, $ (\bds{g}, \bds{f}^\delta) $, $\gamma$, $ \mathcal{M}_{\widehat{\Theta}} $, $\alpha$, $ N_v $

\textbf{Output:} final approximate contrast for $ \bds{m} $

\textbf{Initialize:} $ i = 0 $, $ \bds{m}_0^\delta = 0 $

\begin{algorithmic}[1]
\item Compute the \Calderon's matrix $ \bds{C} $ with the Cauchy data pair $ (\bds{g}, \bds{f}^\delta) $
\item Compute $ \widetilde{\bds{S}} := \mathcal{S}_\gamma(\mathcal{M}_{\widehat{\Theta}}(\mathcal{N}(\bds{C}))) $ and the operator $ \widetilde{\bds{S}}_\alpha $ as \eqref{54}
\item \textbf{while} $ i<N_v  $ \textbf{do}
\item \begin{itemize}
\item[]  $ \bds{m}^\delta_{i+1} =\bds{m}^\delta_i + \left(\wit{\bds{S}}_\alpha + \left[\bds{F}'(\bds{m}^\delta_i \bds{g})\right]^*\bds{F}'(\bds{m}^\delta_i; \bds{g}) \right)^{-1}\left[\bds{F}'(\bds{m}^\delta_i; \bds{g})\right]^*\left(\bds{f}^\delta - \bds{F}(\bds{m}^\delta_i; \bds{g})\right) $
\end{itemize}
\item \begin{itemize}
\item[]  $ i\gets i+1 $	
\end{itemize} 
\item \textbf{end while}
\item Set final approximate contrast to be $\bds{m}_{N_v}^\delta$.
\end{algorithmic}
\end{algorithm}

\begin{remark}\label{R2}\rm
The proposed method \eqref{51} reduces to the classical Tikhonov regularization if the regularization functional $ \mathcal{R}_\alpha := (\alpha/2)\|\cdot\|^2_2 $.
In this case, the Gauss-Newton iteration \eqref{53} simplifies by replacing $ \wit{\bds{S}}_\alpha $ by $ \alpha\bds{I} $, where $ \bds{I} $ denotes the identity operator on $ M_h $.
We refer to this simplified version as the \textit{Tikhonov Algorithm} in this paper. 
To demonstrate the benefits provided by \textbf{LEVR-C}, we will compare the convergence rate of the proposed regularization method and Tikhonov regularization in Appendix \ref{A} under some appropriate assumptions, and the numerical performance of the proposed \textbf{LEVR-C} and \textit{Tikhonov Algorithm} will be carried out in Section \ref{S5}.
\end{remark}

\section{Numerical experiments}\label{S5}
\setcounter{equation}{0}

This section presents numerical examples to demonstrate the effectiveness of the proposed learning-enhanced variational regularization via \Calderon's method (i.e., \textbf{LEVR-C}) for the EIT problem.
Subsection \ref{S51} details the experimental configuration.
As $\mathcal{M}_{\widehat{\Theta}}$ plays a vital role in our method, we evaluate its training performance and generalization ability in Subsection \ref{S52}.
The performance of \textbf{LEVR-C} is shown in Subsection \ref{S53}.
To evaluate how the \textit{a priori} information of the shape and location (i.e., the support) of the unknown contrast contributes to the EIT problem, Subsection \ref{S53} compares our algorithm with \Calderon's method \cite{CAP80}, \textit{Tikhonov Algorithm} (see Remark \ref{R2}) and Deep \Calderon \cite{CSJ23}.
It is known that the regularization parameter has a significant impact on reconstruction quality for most variational regularization methods.
Subsection \ref{S54} investigates the influence of the regularization parameter $\alpha$ in the proposed \textbf{LEVR-C}, using \textit{Tikhonov Algorithm} as a reference baseline.
We also compare their convergence rates in Subsection \ref{S54}.

\subsection{Experimental setup}\label{S51}
The training process is performed on a local server (NVIDIA A100 GPU, Linux) and is implemented on PyTorch, whereas the computations of forward problem and \textbf{LEVR-C} are implemented using EIDORS (version 3.10) \cite{AAL06}.
All EIDORS simulations are run on MATLAB R2022b through a desktop workstation (Intel Core i7-10700 CPU (2.90 GHz), 32 GB of RAM, Windows 11).

\subsubsection{Simulation setup for the EIT problem}\label{S511}
As introduced in Section \ref{S2}, the support of the unknown contrast is assumed to lie within the square domain $ \Omega=[-1,1]\times [-1,1]$. 
We set $ P=32 $ electrodes with $ Q=32 $ excitation current patterns. 
The synthetic Cauchy data $ (\bds{g}, \bds{f}^\delta) $ is generated via finite element discretization of $ M_h $ with $N=320$ (i.e., $ h=1/160 $), as defined in Section \ref{S2}.
The noisy Dirichlet boundary data $ \bds{f}^\delta = (\bds{f}_q) $, $ q = 1,\ldots, Q$, is generated by adding relative Gaussian noise into the exact data pointwise as
\begin{equation}\label{61}
\bds{f}_q^\delta = \bds{f}_q + \delta\|\bds{f}_q\|_{\max}\bds{\xi}_q,
\end{equation}
where $\delta$ is the noise level and $ \bds{\xi}_q $ are standard normal distributions.
Unless otherwise specified, we use $ \delta = 10^{-4} $ as the default noise level.

\subsubsection{Parameter setting for inversion algorithms}\label{S512}

The proposed \textbf{LEVR-C} adopts the following parameters: finite element discretization of $ M_h $ with $ N = 80 $ (i.e., $ h = 2/N = 1/40 $), total iteration number $ N_v = 20 $, regularization parameter $ \alpha=10^{-3} $ and threshold $ \gamma=0.1 $.
The deep neural network $\mathcal{M}_{\widehat{\Theta}}$ is trained for $ t=200 $ epochs with the loss function $\mathcal{L}$ using the Adam optimizer \cite{KDP14} (batch size 100, learning rate of $ 10^{-4} $). 

\Calderon's method is implemented with resolution $ N=80 $ and the truncation radius $ R = 1.4 $. 
The \textit{Tikhonov Algorithm} share the same total iteration number and regularization parameter as our method (i.e., $ N_v=20,\;\alpha=10^{-3} $).
Deep \Calderon\;follows the same data preprocessing and loss function from \cite{CSJ23}, with identical training hyperparameters (batch size/learning rate/epochs) for fair comparison.

\subsubsection{Data generation for the neural network}\label{S513}

In this paper, we consider the exact contrast $ m $ consisting of several disjoint circular conductivity inclusions.
Specifically, let $U(a,b)$ denote the uniform distribution over $ [a,b] $.
We randomly generate two or three disjoint circular regions $ \{c_k\}_{k=1}^K $ defined by
\begin{equation}\label{62}
c_k := \{(x,y)\in\Omega: (x-x_k)^2 + (y-y_k)^2 \leq R^2_k\},
\end{equation}
where
\begin{itemize}
\item Each center $ (x_k,y_k) $ is uniformly distributed within $\Omega = [-1,1]\times [-1,1]$.
\item Each radius $ R_k $ is sampled from $ U(0.15,0.25) $.
\end{itemize} 
The conductivity contrast values are assigned as:
\begin{equation}\label{63}
m(x) = v_k,\quad \forall x\in c_k,\quad v_k\sim U(0,V_k),
\end{equation}
where $ V_k\sim U(1, 3) $.
An exemplary realization is illustrated in Figure \ref{F1} (a). 
This \textit{Circle Dataset} will be used to train the proposed deep neural network $ \mathcal{M}_\Theta $ (see Subsection \ref{S52}).

\subsubsection{Evaluation criterion}\label{S514}

Note that reconstructions by \textbf{LEVR-C} and \textit{Tikhonov Algorithm} belong to the finite element space $ M_h $ (see Section \ref{S2}), whereas \Calderon's method and Deep \Calderon\;produce $ (N\times N) $ pixel images.
For uniform comparison, we define:
\begin{itemize}
\item For any $ \bds{m}\in M_h $, its matrix representation $ \wit{\bds{m}} = (\wit{\bds{m}}_{ij})\in \R^{N\times N} $ with $ \wit{\bds{m}}_{ij} := \bds{m}(x_{ij}) $, where $ x_{ij} $ is the $ (i,j) $-th pixel center of the domain $\Omega$ (see Section \ref{S2}).
\item The ground truth $ \bds{m}^\dagger = (\bds{m}^\dagger_{ij})\in \R^{N\times N} $ with $ \bds{m}^\dagger_{ij} := m(x_{ij}) $, for the exact contrast $ m $.
\end{itemize}
Let $ \wit{\bds{m}} $ be the output of one of the aforementioned inversion algorithms, which is the approximation of $ \bds{m}^\dagger $. 
Then the relative error between the reconstructed conductivity distribution $\wit{\bds{\sigma}} := \wit{\bds{m}} + 1 $ and true conductivity distribution $\bds{\sigma}^\dagger := \bds{m}^\dagger + 1 $ can be defined as:
\begin{equation}\label{64}
\mathcal{E}(\bds{\sigma}^\dagger, \wit{\bds{\sigma}}) := \dfrac{\|\bds{\sigma}^\dagger - \wit{\bds{\sigma}}\|}{\|\bds{\sigma}^\dagger\|},
\end{equation}
where $ \|\cdot\| $ is the Frobenius norm as we defined in Subsection \ref{S322}.

As discussed in Section \ref{S4}, the approximate support matrix $ \wit{\bds{S}} $ provided by the trained neural network $ \mathcal{M}_{\widehat{\Theta}} $ plays an important role in the proposed \textbf{LEVR-C}.
Note that both $ \wit{\bds{S}} = (\wit{\bds{S}}_{ij}) $ and the exact contrast matrix $ \bds{S} = (\bds{S}_{ij}) $ are binary matrices (see  Subsection \ref{S322} and Section \ref{S2}), with $ \mathcal{M}_{\widehat{\Theta}} $ trained specifically for this segmentation task.
To quantify the agreement between $ \bds{S} $ and $ \wit{\bds{S}} $, we adopt three standard image segmentation metrics \cite{WZW20}:
\begin{itemize}
\item $Dice(\bds{S}, \wit{\bds{S}}) := \dfrac{2\sum_{i,j}\bds{S}_{ij}\wit{\bds{S}}_{ij}}{\sum_{i,j}\bds{S}_{ij} + \sum_{i,j}\wit{\bds{S}}_{ij}}$.

\item $Recall(\bds{S}, \wit{\bds{S}}) := \dfrac{\sum_{i,j}\bds{S}_{ij}\wit{\bds{S}}_{ij}}{\sum_{i,j}\bds{S}_{ij}}$.

\item $Precision(\bds{S}, \wit{\bds{S}}) := \dfrac{\sum_{i,j}\bds{S}_{ij}\wit{\bds{S}}_{ij}}{\sum_{i,j}\wit{\bds{S}}_{ij}}$.
\end{itemize}

\subsection{Performance evaluation of $\mathcal{M}_{\widehat{\Theta}}$}\label{S52}

The deep neural network $\mathcal{M}_{\Theta}$ is trained on the \textit{Circle Dataset} (see Subsection \ref{S513}) using the configuration specified in Subsection \ref{S512}, with $ 9000 $ samples for training and $ 1000 $ for validation.

First, we evaluate the training performance of $\mathcal{M}_{\widehat{\Theta}}$.
For this purpose, we compare $\mathcal{M}_{\widehat{\Theta}}$ with Deep \Calderon\ as a baseline method.
Following \cite{CSJ23}, Deep \Calderon\ is trained to map \Calderon's matrix $ \bds{C} $ (see Subsection \ref{S31}) directly to the exact contrast, which can be categorized as a post-processing method \cite{ASM19}.
For fairness, both models use the same training and validation samples during the training process, under identical training configurations (see Subsection \ref{S512}).
Figure \ref{F3} shows the normalized training dynamics, where both the training loss (blue curve) and validation loss (yellow curve) are scaled by their respective maximum training loss values.
As shown in Figure \ref{F3}, the proposed $ \mathcal{M}_{\widehat{\Theta}} $ achieves steeper loss reduction relative to Deep \Calderon, with a consistently smaller training-validation gap.
Regarding this, we believe this is because our training strategy (see Section \ref{S322}) makes $ \mathcal{M}_{\widehat{\Theta}} $ focus on learning the \textit{a priori} information of the shape and location (i.e., the support) of the unknown contrast, whereas Deep \Calderon\ is supposed to additionally learn the value of the unknown contrast. 

Second, we assess the generalization ability of $\mathcal{M}_{\widehat{\Theta}}$.
To do this, we generate $ 100 $ independent test samples from the \textit{Circle Dataset} to represent the exact contrast.
We test the trained $\mathcal{M}_{\widehat{\Theta}}$ on these $ 100 $ samples using the metrics of $ Dice $, $ Recall $ and $ Precision $ (see Subsection \ref{S514}), where the averages and the variances of these three metrics are presented in Table \ref{T1}.
The test results in Table \ref{T1} imply that $\mathcal{M}_{\widehat{\Theta}}$ can provide reliable \textit{a priori} information of the support of the unknown contrast for our \textbf{LEVR-C}.
In particular, the high $ Recall $ score in Table \ref{T1} indicates that the approximate support matrix $ \wit{\bds{S}} $ provided by $\mathcal{M}_{\widehat{\Theta}}$ can approximately satisfy $ \mathrm{supp}(\bds{S})\subset\mathrm{supp}(\wit{\bds{S}}) $ (strictly satisfied when $ Recall $ reaches $ 1 $), which is a crucial requirement for the proposed \textbf{LEVR-C} (see Remark \ref{R3}).

\begin{figure}[!htbp]
\centering
\includegraphics[width=.9\textwidth]{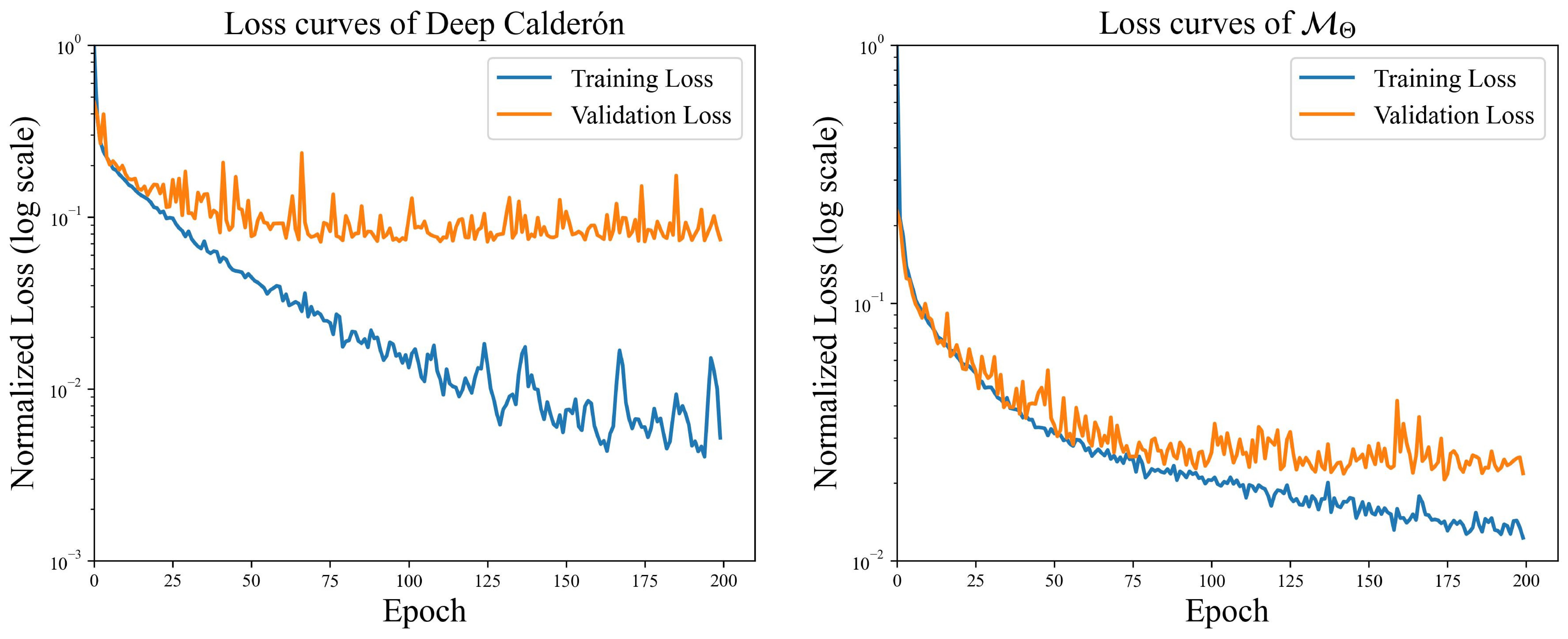}
\caption{Normalized training/validation loss curves of  Deep \Calderon\ (left) and $\mathcal{M}_{\widehat{\Theta}}$  (right).
The blue curve and yellow curve in each subplot represent the training loss curve and validation loss curve, respectively.}\label{F3}
\end{figure}

\begin{table}[!htpb]
\centering
\begin{tabular}{|p{2cm}<{\centering}|p{3cm}<{\centering}|p{3cm}<{\centering}|p{3cm}<{\centering}|}
\hline
& $ Dice [\%] $ & $ Recall [\%] $ & $ Precision [\%] $\\
\hline
average$ \uparrow $ & 97.50 & 98.80 & 96.33 \\
\hline 
variance$ \downarrow $ & $ 1.64\times 10^{-4} $ & $ 3.67\times 10^{-4} $ & $ 3.95\times 10^{-4} $ \\
\hline
\end{tabular}
\caption{The averages and variances of $ Dice $, $ Recall $ and $ Precision $ for the approximate support matrix provided by $\mathcal{M}_{\widehat{\Theta}}$ on \textit{Circle Dataset}.}\label{T1}
\end{table}

\subsection{Performance of \textbf{LEVR-C}}\label{S53}

For all inversion algorithms, we use the neural networks trained in Subsection \ref{S52} with parameter configurations specified in Subsection \ref{S512}.

First, we investigate the influence of values of different contrasts on \Calderon's method, \textit{Tikhonov Algorithm}, Deep \Calderon\ and \textbf{LEVR-C} quantitatively.
To do this, we consider three cases named Cases 1.1, 1.2 and 1.3.
In all three cases, we generate 100 independent test samples from the \textit{Circle Dataset} to represent the exact contrast matrices.
For Cases 1.1, 1.2 and 1.3, we set the norm $\|\cdot\|_{\max}$ of each exact contrast matrix to be 2, 3 and 4, respectively.
Table \ref{T1} displays the average relative errors $\mathcal{E}$ for all four algorithms across three cases.
Figure \ref{F4} presents the reconstruction results of several samples from these three cases. 
Each row of Figure \ref{F4} presents the numerical reconstructions of these four algorithms and the ground truth for one sample. 
Reconstruction results in Table \ref{T2} and Figure \ref{F4} show that our \textbf{LEVR-C} outperforms \Calderon's method and the \textit{Tikhonov Algorithm}, which shows the advantage offered by the \textit{a priori} information of the shape and location (i.e., the support) of the unknown contrast learned by $\mathcal{M}_{\widehat{\Theta}}$.
It can also be observed in Table \ref{T2} and Figure \ref{F4} that the proposed \textbf{LEVR-C} has a better performance than Deep \Calderon, especially in high contrast case.
Regarding this, we believe that our training strategy makes $\mathcal{M}_{\widehat{\Theta}}$ focus on learning the \textit{a priori} information of the support of the unknown contrast, which helps $\mathcal{M}_{\widehat{\Theta}}$ generalize to the high contrast case.
In contrast, Deep \Calderon\ tries to learn both the support and contrast values, which makes it less effective for the high contrast case.
Additionally, our method also uses the EIT physical model more effectively during the reconstruction process, compared with Deep \Calderon.
Moreover, the proposed \textbf{LEVR-C} still has satisfactory performance even when test contrasts are stronger than training examples (see Case 1.3).
For this observation, we believe this is because our deep neural network $\mathcal{M}_{\widehat{\Theta}}$ could extract the \textit{a priori} information of the  support from the \Calderon's method for various values of different contrasts, which is critical \textit{a priori} information for high quality reconstruction.

Secondly, we test the generalization ability of the proposed algorithm with the exact contrast matrix outside \textit{Circle Dataset}. 
The corresponding reconstructions of \Calderon's method, \textit{Tikhonov Algorithm}, Deep \Calderon\ and \textbf{LEVR-C} are shown in Figure \ref{F5}, where each row presents the approximate support generated by $ \mathcal{M}_{\widehat{\Theta}} $, the numerical reconstructions by these four algorithms and the ground truth for one sample. 
The reconstructions in Figure \ref{F5} demonstrate the good generalization ability of our algorithms, which could recover unknown contrasts with various shapes and values, even when they are not sampled from the \textit{Ellipse Dataset}. 
It is worth noting that our algorithms can also recover general piecewise smooth contrasts, whereas the training samples from \textit{Circle Dataset} are piecewise constant.
Moreover, it can be observed from Figure \ref{F5} that our method could generate satisfactory reconstruction results even when $ \mathcal{M}_{\widehat{\Theta}} $ does not provide accurate \textit{a priori} information of the support of the unknown contrast.

\begin{table}[htpb]
\centering
\begin{tabular}{|p{5.5cm}<{\centering}|p{2cm}<{\centering}|p{2cm}<{\centering}|p{2cm}<{\centering}|}
\hline
& Case 1.1 & Case 1.2 & Case 1.3\\
\hline
\Calderon's method & 37.60\% & 49.84\% & 58.79\% \\
\hline 
\textit{Tikhonov Algorithm} & 20.46\% & 28.52\% & 35.15\% \\
\hline
Deep \Calderon & 13.28\% & 18.72\% & 26.84\%\\
\hline
\textbf{LEVR-C} (ours) & 9.56\% & 13.22\% & 15.74\%\\
\hline
\end{tabular}
\caption{The average values of relative errors $ \mathcal{E} $ for the outputs of \Calderon's method, \textit{Tikhonov Algorithm}, Deep \Calderon\ and \textbf{LEVR-C} on the \textit{Circle Dataset}.}\label{T2}
\end{table}

\begin{figure}[htbp]
\centering
\includegraphics[width=1.\textwidth]{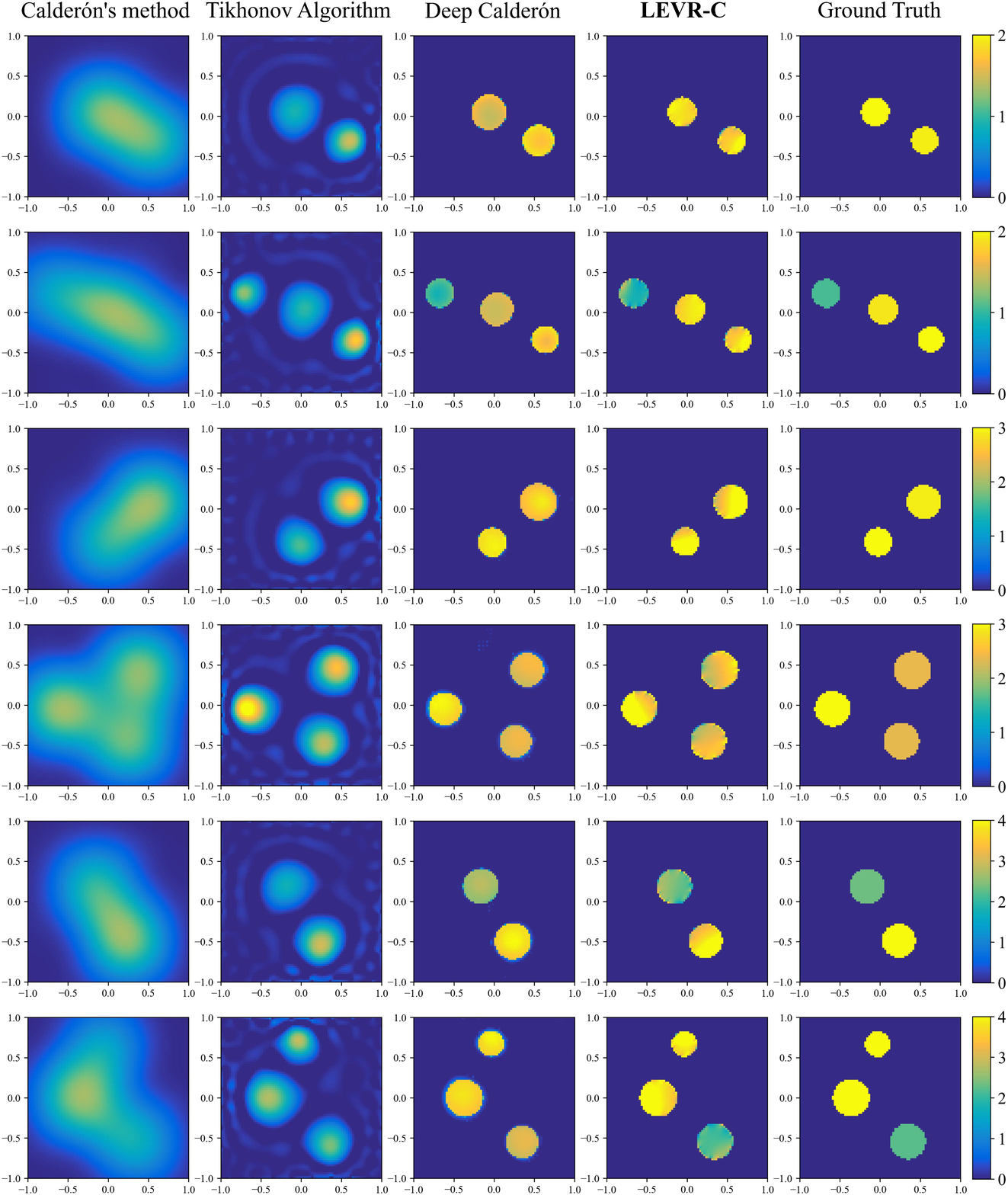}
\caption{Numerical results by \Calderon's method, \textit{Tikhonov Algorithm}, Deep \Calderon\ and \textbf{LEVR-C} for the exact contrast matrices generated from \textit{Circle Dataset}. Each row presents the reconstructions by \Calderon's method, \textit{Tikhonov Algorithm}, Deep \Calderon, \textbf{LEVR-C} and the ground truth for one sample.}\label{F4}
\end{figure}

\begin{figure}[htbp]
\centering
\includegraphics[width=1.\textwidth]{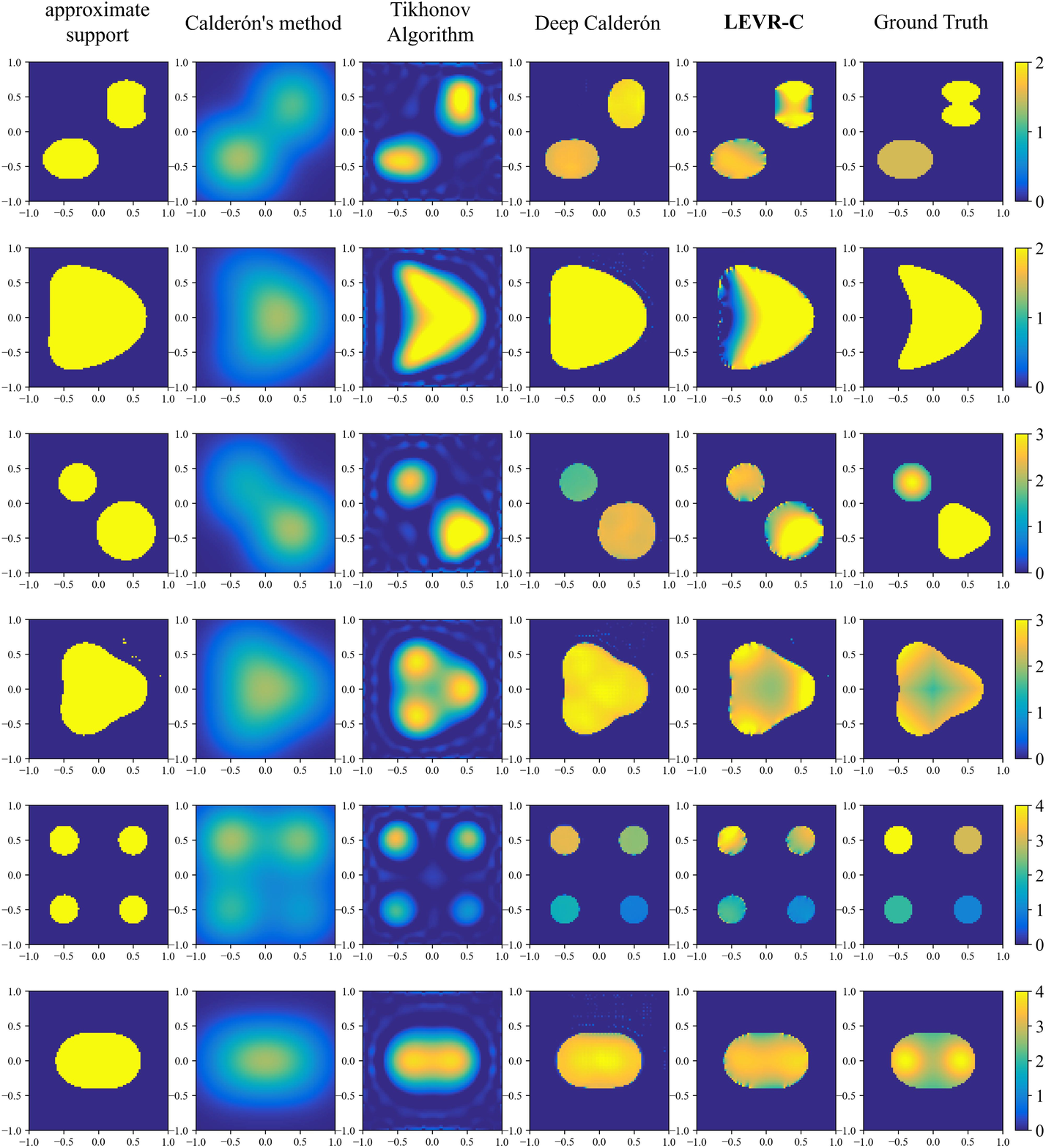}
\caption{Numerical results by \Calderon's method, \textit{Tikhonov Algorithm}, Deep \Calderon\ and \textbf{LEVR-C} for the exact contrast matrices generated from \textit{Circle Dataset}. Each row presents the approximate support generated by $ \mathcal{M}_{\widehat{\Theta}} $, the reconstructions by \Calderon's method, \textit{Tikhonov Algorithm}, Deep \Calderon, \textbf{LEVR-C} and the ground truth for one sample.}\label{F5}
\end{figure}

\subsection{Regularization parameter analysis}\label{S54}

In order to study the influence of the regularization parameter $\alpha>0$ (see Section \ref{S4}) on the proposed \textbf{LEVR-C}, we test our algorithm on two numerical samples with various values of $ \alpha $.
In this subsection, the influence of $\alpha$ on \textit{Tikhonov Algorithm} is also studied as a baseline.
The reconstruction results of \textbf{LEVR-C} and \textit{Tikhonov Algorithm} for the two samples are shown in Figures \ref{F6} and \ref{F7}.
To be more specific, each of Figure \ref{F6} (a) and Figure \ref{F7} (a) displays the reconstructions of \textbf{LEVR-C} with their corresponding relative error $\mathcal{E}$ below and regularization parameter $ \alpha $ above, and the ground truth for one sample.
Each of Figure \ref{F6} (b) and Figure \ref{F7} (b) presents the error curves of $ \mathcal{E}^+ $ (left) and $ \mathcal{E}^- $ (right) of \textbf{LEVR-C} (yellow curve) and \textit{Tikhonov Algorithm} (blue curve) for one sample, as $ \alpha\to 0 $, where
\begin{equation}\label{65}
\mathcal{E}^+(\bds{m}^\dagger, \wit{\bds{m}}) := \|\widetilde{\bds{S}}\odot(\bds{m}^\dagger-\wit{\bds{m}})\|,\quad \mathcal{E}^-(\bds{m}^\dagger, \wit{\bds{m}}) := \|(1-\widetilde{\bds{S}})\odot(\bds{m}^\dagger-\wit{\bds{m}})\|,
\end{equation}
where $ \odot $ denotes the element-wise multiplication of two matrices of the same size, and $ \bds{m}^\dagger $ and $ \wit{\bds{m}} $ are the same as in Section \ref{S514}.
Here $ \widetilde{\bds{S}} := \mathcal{S}_\gamma(\mathcal{M}_{\widehat{\Theta}}(\mathcal{N}(\bds{C}))) $ is the approximate support generated by our deep neural network $ \mathcal{M}_{\widehat{\Theta}} $ (see Subsection \ref{S322}).
We note that $ \mathcal{E}^+ $ characterizes the reconstruction error of $ \wit{\bds{m}} $ inside $ \mathrm{supp}(\widetilde{\bds{S}}) $, whereas $ \mathcal{E}^- $ characterizes the reconstruction error outside $ \mathrm{supp}(\widetilde{\bds{S}}) $.
The functional interpretation of $ \wit{\bds{S}} $ follows from Subsection \ref{S322}.
It can be observed from Figures \ref{F6} and \ref{F7} that the proposed \textbf{LEVR-C} is not very sensitive to its regularization parameter $\alpha$, as the relative errors $\mathcal{E}$ of reconstruction results do not grow drastically as $ \alpha\to 0 $.
Further, the error curves in Figures \ref{F6} (b) and \ref{F7} (b) show that the solutions converge faster in the area outside $ \mathrm{supp}(m) $, compare to the area inside $ \mathrm{supp}(m) $, which is consistent with the implication of Theorem \ref{thm54}.
Moreover, the reconstruction results in Figures \ref{F6} and \ref{F7} demonstrate the importance of the \textit{a priori} information of the support of the unknown contrast for the EIT problem.

\begin{figure}[htbp]
\centering
\includegraphics[width=1.\textwidth]{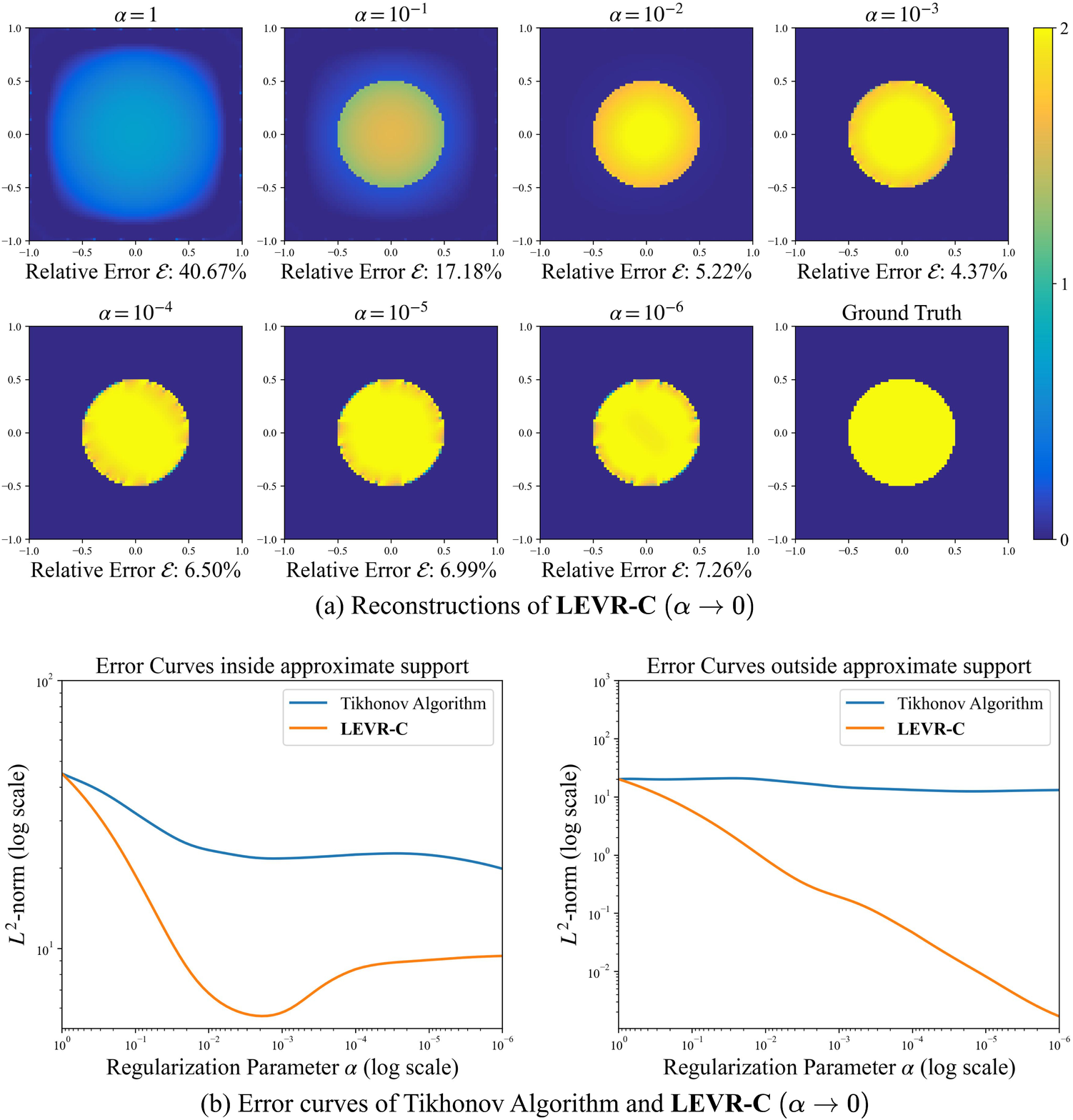}
\caption{Numerical results by \textit{Tikhonov Algorithm} and \textbf{LEVR-C} for one sample ($\alpha\to 0$).}\label{F6}
\end{figure}

\begin{figure}[htbp]
\centering
\includegraphics[width=1.\textwidth]{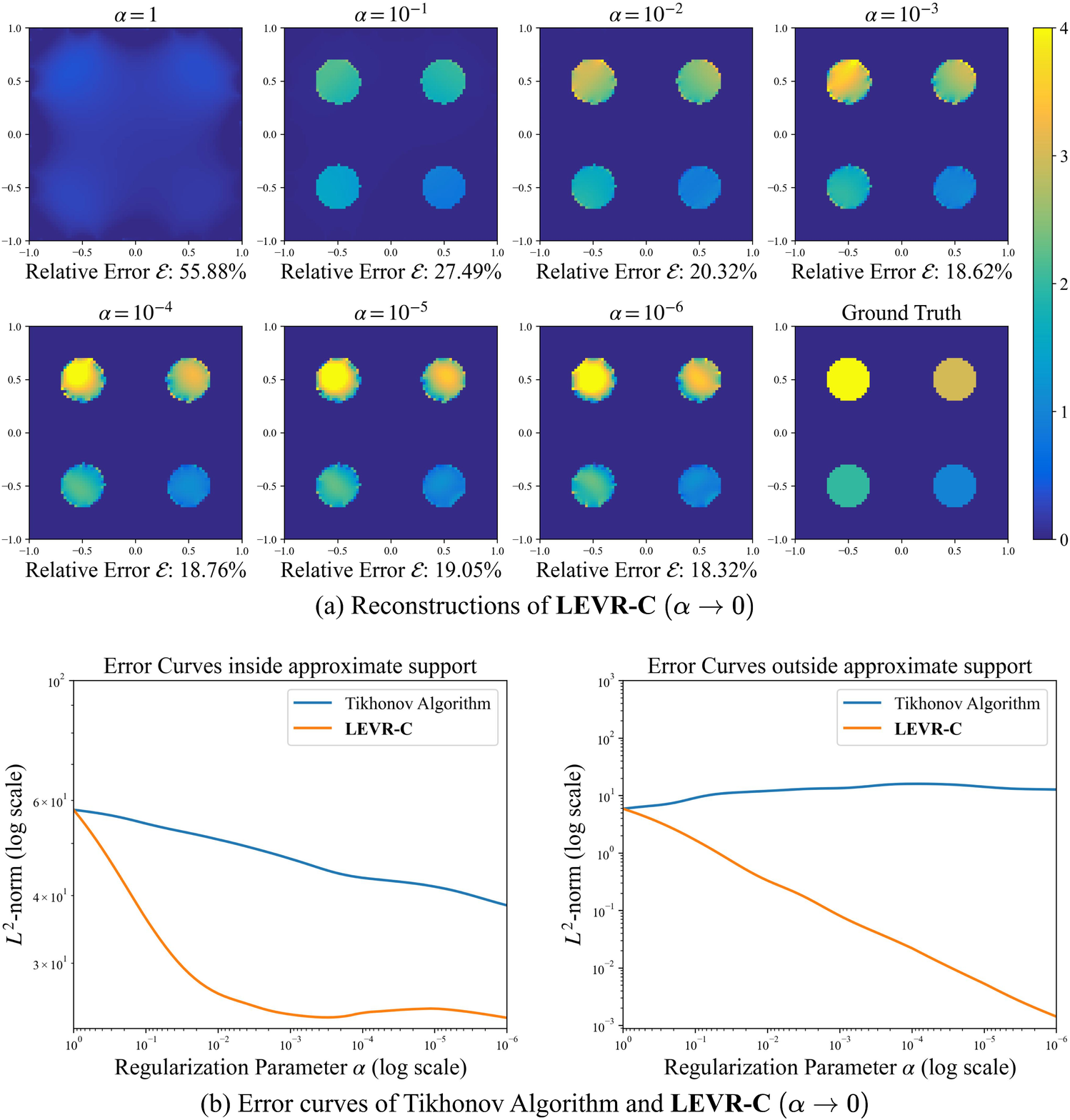}
\caption{Numerical results by \textit{Tikhonov Algorithm} and \textbf{LEVR-C} for one sample ($\alpha\to 0$).}\label{F7}
\end{figure}

\section{Conclusion}\label{S6}
\setcounter{equation}{0}
In this work, we addressed the inverse problem of electrical impedance tomography (EIT) in two dimensions, namely the reconstruction of the conductivity distribution (or its contrast) from its Cauchy data.
This EIT problem is highly nonlinear and severely ill-posed, and an appropriate regularization strategy, which relies heavily on certain \textit{a priori} information of the unknown contrast, is thus needed.
To this end, we trained a deep neural network $ \mathcal{M}_{\wih{\Theta}} $ to retrieve the \textit{a priori} information of the shape and location (i.e., the support) of the unknown contrast from \Calderon's method.
We then proposed the learning-enhanced variational regularization via \Calderon's method (i.e., \textbf{LEVR-C}) that incorporates the learned \textit{a priori} information of the support of the unknown contrast, and subsequently used the Gauss-Newton method to solve the learned variational formulation for EIT reconstruction.
The stability and convergence analysis of our regularization method demonstrates the significance of the \textit{a priori} information of the support of the unknown contrast.
With the aid of the \textit{a priori} information of the unknown contrast provided by the deep neural network $ \mathcal{M}_{\wih{\Theta}} $, our \textbf{LEVR-C} can provide good reconstruction results even for the high contrast case.

Extensive numerical experiments demonstrate that \textbf{LEVR-C} performs well for the inverse problem of EIT. 
First, we observe that our reconstruction algorithms could generate satisfactory results for a large variety of contrasts with different values and even for the values of the contrasts for testing are higher than those for training. 
For this observation, it is reasonable to deduce that $ \mathcal{M}_{\widehat{\Theta}} $ can extract the \textit{a priori} information of the shape and location (i.e., the support) from the \Calderon's method for various values of different contrasts, which is critical \textit{a priori} information for reconstruction.
Secondly, we notice that the proposed \textbf{LEVR-C} has good performance on samples outside the \textit{Circle Dataset}, on which we train our deep neural network $ \mathcal{M}_{\widehat{\Theta}} $. 
This suggests the satisfactory generalization ability of our algorithms, which can recover unknown contrasts with various shapes and values.
Thirdly, we observe that our algorithm does not show a significant performance decrease as the regularization parameter $\alpha\to 0$.
Regarding this, we believe this is because the \textit{a priori} information of the support of the unknown contrast in our regularization method can boost the convergence rate as $\alpha\to 0$.
However, it is observed in the numerical experiments that the reconstruction quality of \textbf{LEVR-C} becomes worse when the value of the exact contrast is large. 
One of the reasons may be due to the fact that when the value of the exact contrast is large the numerical solution of the corresponding forward problem in each iterative step is not accurate enough, which will deteriorate the final reconstruction results. 
It is also interesting to extend our method to other challenging inverse problems, which will be considered as a future work.

\appendix
\section{Variational regularization method based on support information}\label{A}

In this appendix, we establish the stability and convergence of the learning-enhanced variational regularization method (i.e., \textbf{LEVR-C}) \eqref{51} under some appropriate assumptions, and thus confirm the learned \textit{a priori} information of the support of the unknown contrast provided by the deep neural network $ \mathcal{M}_{\wih{\Theta}} $ (see Section \ref{S3}).
It should be noted that our variational regularization method can be generalized to various inverse problems.
Precisely, we first consider an inverse problem of a general form:
\begin{equation}\label{31}
F(x) = y,
\end{equation}
where $ y $ is the measured data, $ x $ is the unknown function we aim to recover and $ F: \mathcal{D}(F)\subset L^2(\Omega)\to \mathcal{Y} $ (possibly nonlinear) is the forward operator of the inverse problem.
Here, we assume $ \mathcal{D}(F)\subset L^2(\Omega) $ denotes the domain of the forward operator $ \mathcal{F} $, $ \mathcal{Y} $ is some appropriate Hilbert space and (with a slight abuse of notation) $ \Omega\subset \mathbb{R}^n $ is measurable and bounded.
Throughout this appendix, we assume that
\begin{itemize}
\item 
$ F $ is continuous and
\item 
$ F $ is weakly (sequentially) closed.
\end{itemize}
These hold when $ F $ is continuous and compact, with $ \mathcal{D}(F) $ weakly closed (e.g., closed and convex).
The aim of solving the inverse problem \eqref{31} is to recover the exact solution $ x $ from measurements $ y $, i.e., solve the (nonlinear) equation \eqref{31}.
To derive the variational regularization method that incorporates the \textit{a priori} information of the support of the exact solution, we assume that an exact solution $ x^\dagger\in\mathcal{D}\subset L^2(\Omega) $ to \eqref{31} exists, with $\mathcal{S}\supset\mathrm{supp}(x^\dagger)$ for some measurable $ \mathcal{S}\subset \Omega $. 
Given such \textit{a priori} information of $ \mathrm{supp}(x^\dagger) $, we define the $ x^* $-$\mathcal{S}$-minimum-norm solution $ \hat{x} $ for the inverse problem \eqref{31} as follows:
\begin{displaymath}
\hat{x} := \underset{x\in \mathcal{D}(F)}{\arg\min}\left\{\left\|x - x^*\right\|_2: F(x) = y\;\text{subject to}\; \mathrm{supp}(x)\subset\mathcal{S}\right\},
\end{displaymath}
where $ x^* $ is an initial guess satisfying $ \mathrm{supp}(x^*)\subset\mathcal{S} $.
Here, $ \left\|\cdot\right\|_2 $ denotes the $ L^2 $-norm of a function.
Moreover, we require $ x^* $ to be sufficiently close to $ x^\dagger $, ensuring $ x^\dagger $ is an $ x^* $-$\mathcal{S}$-minimum-norm solution of the inverse problem \eqref{31} (trivially satisfied when $ x^* = x^\dagger $).
Given the regularization parameter $ \alpha>0 $ and $\mathcal{S}\subset\Omega$, we define                               
\begin{align*}
f_{\mathcal{S},\alpha}(\lambda) := \left\{
\begin{aligned}
&\sqrt{\alpha}, &&\lambda\in \mathcal{S},\\
&1, && \lambda\in \Omega\setminus\mathcal{S}.
\end{aligned}
\right.
\end{align*}
The proposed variational regularization method for the inverse problem \eqref{31} is:
\begin{equation}\label{32}
\min_{x\in \mathcal{D}(F)}\left\{\left\|F(x) - y^\delta\right\|^2_\mathcal{Y} + \left\|x - x^*\right\|_{\mathcal{S}, \alpha}^2\right\},
\end{equation}
where $ \|\cdot\|_\mathcal{Y} $ is the norm in the Hilbert space $ \mathcal{Y} $ and $ y^\delta\in\mathcal{Y} $ is the noisy measurement satisfying $ \|y^\delta-y\|_\mathcal{Y}\leq \delta $. 
Here, for all $ x\in L^2(\Omega) $, the norm $ \|\cdot\|_{\mathcal{S}, \alpha} $ is defined by
\begin{displaymath}
\left\|x\right\|_{\mathcal{S}, \alpha} := \left\|f_{\mathcal{S},\alpha}\cdot x\right\|_2 = \left(\alpha\int_{\mathcal{S}}|x(\lambda)|^2d\lambda + \int_{\Omega\setminus\mathcal{S}}|x(\lambda)|^2d\lambda\right)^{1/2}.
\end{displaymath}
For the convenience of later use, we further give the definitions of the following norms:
\begin{displaymath}
\left\|x\right\|_{\mathcal{S}^+, 2} := \left(\int_{\mathcal{S}}|x(\lambda)|^2d\lambda \right)^{1/2},\qquad \left\|x\right\|_{\mathcal{S}^-, 2} := \left(\int_{\Omega\setminus\mathcal{S}}|x(\lambda)|^2d\lambda \right)^{1/2},
\end{displaymath}
and note that
\begin{displaymath}
\left\|x\right\|_{\mathcal{S}, \alpha}^2 = \alpha\left\|x\right\|_{\mathcal{S}^+, 2}^2 + \left\|x\right\|_{\mathcal{S}^-, 2}^2,\quad \left\|x\right\|_2^2 = \left\|x\right\|_{\mathcal{S}^+, 2}^2 + \left\|x\right\|_{\mathcal{S}^-, 2}^2.
\end{displaymath}
Under the given assumptions on the forward operator $ F $, the variational problem \eqref{32} admits at least one solution. 
Since $ F $ is potentially nonlinear, the solution is not necessarily unique.
It should be noted that if the exact discrete contrast $ \mathrm{supp}(\bds{m})= \mathrm{supp}(\bds{S})\subset \mathrm{supp}(\wit{\bds{S}}) $ (see Subsection \ref{S322}), then the application of \eqref{52} on the EIT problem yields the proposed learning-enhanced variational regularization \eqref{51} (i.e., \textbf{LEVR-C}).
As noted in Remark \ref{R2}, the weighted norm $ \left\|\cdot\right\|_{\mathcal{S}, \alpha} $ in \eqref{32} reduces to a scaled $ L^2 $-norm $ \alpha\left\|\cdot\right\|_2 $ when $ \mathcal{S}=\Omega $, recovering classical Tikhonov regularization.
However, we highlight that the proposed regularization functional $ \left\|\cdot\right\|_{\mathcal{S}, \alpha} $ is equipped with the \textit{a priori} information of the support of the exact solution, distinguishing it from classical Tikhonov regularization.
Precisely, our regularization functional $ \left\|\cdot\right\|_{\mathcal{S}, \alpha} $ is expected to penalize large values of an approximate solution outside $\mathcal{S}$ as $\delta,\; \alpha\to 0$, and hence can hopefully lead to good reconstruction results if $\mathcal{S}$ is a good approximation of $ \mathrm{supp}(x^\dagger)$ (see Proposition \ref{P1}).
Moreover, Theorem \ref{thm54} suggests that our variational regularization method has a better convergence rate, compared with the classical Tikhonov regularization.
In what follows, we present the stability analysis and convergence analysis of the proposed variational regularization method in Subsections \ref{A1} and \ref{A2}, respectively.

\begin{remark}\rm
The variational regularization method \eqref{32} requires the \textit{a priori} information of a measurable set $\mathcal{S}$ satisfying $ \mathrm{supp}(x^\dagger)\subset\mathcal{S} $.
However, the support information is generally unavailable for EIT problem and most inverse problems.
In this paper, we address this issue by learning the \textit{a priori} information of the support from \Calderon's method via the well-trained $\mathcal{M}_{\wih{\Theta}}$, yielding a practical $\mathcal{S}$ (see Section \ref{S3}).
\end{remark}

\subsection{Stability analysis}\label{A1}

We first prove that the problem of solving \eqref{32} is stable in the sense of continuous dependence of the solutions on measurement $ y^\delta $ (measurement stability).
\begin{theorem}[measurement stability]\label{thm51}
Let $ \alpha>0 $ and let $ \{y_k\} $ and $ \{x_k\} $ be sequences where $ y_k\to y^\delta $ and $ x_k $ is a minimizer of \eqref{32} with $ y^\delta $ replaced by $ y_k $. Then there exists a convergent subsequence of $ x_k $ and the limit of every convergent subsequence is a minimizer of \eqref{32}.
\end{theorem}
\begin{proof}
By the definition of $ x_k $ we have
\begin{equation}\label{33}
\left\|F(x_k) - y_k\right\|^2_\mathcal{Y} + \left\|x_k - x^*\right\|_{\mathcal{S}, \alpha}^2\\
\leq 
\left\|F(x) - y_k\right\|^2_\mathcal{Y} + \left\|x - x^*\right\|_{\mathcal{S}, \alpha}^2
\end{equation}
for any $ x\in \mathcal{D}(F) $. Hence, $ \{\left\|x_k\right\|_2\} $ and $ \{\left\|F(x_k)\right\|_\mathcal{Y}\} $ are bounded. Therefore, a subsequence $ \{x_m\} $ of $ \{x_k\} $ and $\bar{x}$ exist such that
\begin{displaymath}
x_m\rightharpoonup\bar{x}\quad \mathrm{and}\quad F(x_m)\rightharpoonup F(\bar{x}).
\end{displaymath}
By weak lower semicontinuity of the norm we have
\begin{equation}\label{34}
\left\|\bar{x} - x^*\right\|_{\mathcal{S}, \alpha}\leq \underset{m\to \infty}{\lim\inf}\left\|x_m - x^*\right\|_{\mathcal{S}, \alpha}
\quad \mathrm{and}\quad
\left\|F(\bar{x}) - y^\delta\right\|_\mathcal{Y}\leq \underset{m\to \infty}{\lim\inf}\left\|F(x_m) - y_m\right\|_\mathcal{Y}.
\end{equation}
Moreover, \eqref{33} implies that
\begin{displaymath}
\begin{aligned}
\left\|F(\bar{x}) - y^\delta\right\|^2_\mathcal{Y} + \left\|\bar{x} - x^*\right\|_{\mathcal{S}, \alpha}^2
\leq 
&\underset{m\to \infty}{\lim\inf}\left(\left\|F(x_m) - y_m\right\|^2_\mathcal{Y} + \left\|x_m - x^*\right\|_{\mathcal{S}, \alpha}^2\right)\\
\leq 
&\underset{m\to \infty}{\lim\sup}\left(\left\|F(x_m) - y_m\right\|^2_\mathcal{Y} + \left\|x_m - x^*\right\|_{\mathcal{S}, \alpha}^2\right)\\
\leq
&\lim_{m\to\infty}\left(\left\|F(x) - y_m\right\|^2_\mathcal{Y} + \left\|x - x^*\right\|_{\mathcal{S}, \alpha}^2\right)\\
=
&\left\|F(x) - y^\delta\right\|^2_\mathcal{Y} + \left\|x - x^*\right\|_{\mathcal{S}, \alpha}^2
\end{aligned}
\end{displaymath}
for all $ x\in \mathcal{D}(F) $. This together with \eqref{33} imply that $\bar{x}$ is a minimizer of \eqref{32} and that
\begin{equation}\label{35}
\lim_{m\to\infty}\left(\left\|F(x_m) - y_m\right\|^2_\mathcal{Y} + \left\|x_m - x^*\right\|_{\mathcal{S}, \alpha}^2\right)\\
=
\left\|F(\bar{x}) - y^\delta\right\|^2_\mathcal{Y} + \left\|\bar{x} - x^*\right\|_{\mathcal{S}, \alpha}^2
\end{equation}
Now we assume that $ x_m\not\to\bar{x} $. Then without loss of generality, we assume $ c := \underset{m\to \infty}{\lim\sup}\left\|x_m - x^*\right\|_{\mathcal{S}, \alpha} > \left\|\bar{x} - x^*\right\|_{\mathcal{S}, \alpha} $ and there exists a subsequence $ \{x_n\} $ of $ x_m $ such that $ x_n\rightharpoonup \bar{x} $, $ F(x_n)\rightharpoonup F(\bar{x}) $ and $ \left\|x_n - x^*\right\|_{\mathcal{S}, \alpha}\to c $.
As a consequence of \eqref{35}, we obtain
\begin{equation*}
\begin{aligned}
\lim_{n\to\infty}\left\|F(x_n) - y_n\right\|^2_\mathcal{Y} =
&\left\|F(\bar{x}) - y^\delta\right\|^2_\mathcal{Y} + \alpha\left(\left\|\bar{x} - x^*\right\|_{\mathcal{S}, \alpha}^2 - c^2\right)\\
<&\left\|F(\bar{x}) - y^\delta\right\|^2_\mathcal{Y}
\end{aligned}
\end{equation*}
in contradiction to \eqref{34}. This argument shows that $ x_m\to \bar{x} $.
\end{proof}

The following theorem indicates that the problem of solving \eqref{32} is stable in the sense of continuous dependence of the solutions on $\mathcal{S}\supset \mathrm{supp}(x^\dagger)$ (regularization stability).

\begin{theorem}[regularization stability]\label{thm52}
Let $ \alpha>0 $ and let $ \{\mathcal{S}_k\} $ and $ \{x_k\} $ be sequences where the regularization function $ f_{\mathcal{S}_k,\alpha}\to f_{\mathcal{S},\alpha} $ in $ L^2(\Omega) $ and $ x_k $ is a minimizer of \eqref{32} with $ \mathcal{S} $ replaced by $ \mathcal{S}_k $. 
Then there exists a convergent subsequence of $ x_k $ and the limit of every convergent subsequence is a minimizer of \eqref{32}.
\end{theorem}

\begin{proof}
By the definition of $ x_k $ we have
\begin{equation}\label{37}
\begin{aligned}
\left\|F(x_k) - y^\delta\right\|^2_\mathcal{Y} + \left\|x_k - x^*\right\|_{\mathcal{S}, \alpha}^2
=
&\left\|F(x_k) - y^\delta\right\|^2_\mathcal{Y} + \left\|f_{\mathcal{S}_k,\alpha}\cdot (x_k-x^*)\right\|_2^2\\
\leq 
&\left\|F(x) - y^\delta\right\|^2_\mathcal{Y} + \left\|f_{\mathcal{S}_k,\alpha}\cdot (x-x^*)\right\|_2^2\\
\leq 
&\left\|F(x) - y^\delta\right\|^2_\mathcal{Y} + \left\|x-x^*\right\|_2^2
\end{aligned}
\end{equation}
for any $ x\in \mathcal{D}(F) $. Hence, $ \{\left\|x_k\right\|_2\} $ and $ \{\left\|F(x_k)\right\|_\mathcal{Y}\} $ are bounded. Therefore, a subsequence $ \{x_m\} $ of $ \{x_k\} $ and $\bar{x}$ exist such that
\begin{displaymath}
x_m\rightharpoonup\bar{x}\quad \mathrm{and}\quad F(x_m)\rightharpoonup F(\bar{x}).
\end{displaymath}
By weak lower semicontinuity of the norm we have
\begin{displaymath}
\left\|\bar{x} - x^*\right\|_{\mathcal{S}, \alpha} = \left\|f_{\mathcal{S},\alpha}\cdot (\bar{x}-x^*)\right\|_2^2\\
\leq
\underset{m\to \infty}{\lim\inf}\left\|f_{\mathcal{S}_m,\alpha}\cdot (x_m-x^*)\right\|_2^2 = \underset{m\to \infty}{\lim\inf}\left\|x_m - x^*\right\|_{\mathcal{S}_m, \alpha},
\end{displaymath}
and
\begin{equation}\label{38}
\left\|F(\bar{x}) - y^\delta\right\|_\mathcal{Y}\leq \underset{m\to \infty}{\lim\inf}\left\|F(x_m) - y^\delta\right\|_\mathcal{Y}.
\end{equation}
Moreover, \eqref{37} implies that
\begin{displaymath}
\begin{aligned}
\left\|F(\bar{x}) - y^\delta\right\|^2_\mathcal{Y} + \left\|\bar{x} - x^*\right\|_{\mathcal{S}, \alpha}^2
\leq 
&\underset{m\to \infty}{\lim\inf}\left(\left\|F(x_m) - y^\delta\right\|^2_\mathcal{Y} + \left\|x_m - x^*\right\|_{\mathcal{S}_m, \alpha}^2\right)\\
\leq 
&\underset{m\to \infty}{\lim\sup}\left(\left\|F(x_m) - y^\delta\right\|^2_\mathcal{Y} + \left\|x_m - x^*\right\|_{\mathcal{S}_m, \alpha}^2\right)\\
\leq
&\lim_{m\to\infty}\left(\left\|F(x) - y^\delta\right\|^2_\mathcal{Y} + \left\|x - x^*\right\|_{\mathcal{S}_m, \alpha}^2\right)\\
=
&\left\|F(x) - y^\delta\right\|^2_\mathcal{Y} + \left\|x - x^*\right\|_{\mathcal{S}, \alpha}^2
\end{aligned}
\end{displaymath}
for all $ x\in \mathcal{D}(F) $. This argument shows that $\bar{x}$ is a minimizer of \eqref{32} and that
\begin{equation}\label{39}
\lim_{m\to\infty}\left(\left\|F(x_m) - y^\delta\right\|^2_\mathcal{Y} + \left\|x_m - x^*\right\|_{\mathcal{S}_m, \alpha}^2\right) = \left\|F(\bar{x}) - y^\delta\right\|^2_\mathcal{Y} + \left\|\bar{x} - x^*\right\|_{\mathcal{S}, \alpha}^2.
\end{equation}
Now we assume that $ x_m\not\to\bar{x} $, then $ f_{\mathcal{S}_m,\alpha}\cdot(x_m-x^*)\not\to f_{\mathcal{S},\alpha}\cdot(\bar{x}-x^*) $ but $ f_{\mathcal{S}_m,\alpha}\cdot(x_m-x^*)\rightharpoonup f_{\mathcal{S},\alpha}\cdot(\bar{x}-x^*) $.
Without loss of generality, we assume that $ c := \underset{m\to \infty}{\lim\sup}\left\|f_{\mathcal{S}_m,\alpha}\cdot(x_m-x^*)\right\|_2 > \left\|f_{\mathcal{S},\alpha}\cdot(\bar{x}-x^*)\right\|_2 $ and there exists a subsequence $ \{x_n\} $ of $ x_m $ such that $ f_{\mathcal{S}_n,\alpha}\cdot(x_n-x^*)\rightharpoonup f_{\mathcal{S},\alpha}\cdot(\bar{x}-x^*), F(x_n)\rightharpoonup F(\bar{x}) $ and $ \left\|x_n-x^*\right\|_{\mathcal{S}_n, \alpha} = \left\|f_{\mathcal{S}_n,\alpha}\cdot(x_n-x^*)\right\|_2\to c $. As a sequence of \eqref{39}, we obtain
\begin{equation*}\label{310}
\begin{aligned}
\lim_{n\to\infty}\left\|F(x_n) - y^\delta\right\|^2_\mathcal{Y} =
&\left\|F(\bar{x}) - y^\delta\right\|^2_\mathcal{Y} + \alpha\left(\left\|\bar{x} - x^*\right\|_{\mathcal{S}, \alpha}^2 - c^2\right)\\
=&\left\|F(\bar{x}) - y^\delta\right\|^2_\mathcal{Y} + \alpha\left(\left\|f_{\mathcal{S},\alpha}\cdot(\bar{x}-x^*)\right\|_2^2 - c^2\right)\\
<&\left\|F(\bar{x}) - y^\delta\right\|^2_\mathcal{Y}
\end{aligned}
\end{equation*}
in contradiction to \eqref{38}.
This argument shows that $ x_m\to \bar{x} $.
\end{proof}

\begin{remark}\label{R3}\rm
While deep neural networks are inherently black-box models, the approximate support matrix $ \wit{\bds{S}} $ provided by $\mathcal{M}_{\wih{\Theta}}$ may violate the assumption $ \mathrm{supp}(\bds{S})\subset\mathrm{supp}(\wit{\bds{S}}) $ of the variational regularization method \eqref{32} (see Subsection \ref{S322} for more details).
Fortunately, Theorem \ref{thm52} establishes the solution's continuous dependence on $ \mathrm{supp}(\bds{S}) $.
This allows using $ \wit{\bds{S}} $ in our method for solving the EIT problem, provided it approximates the exact support matrix $ \bds{S} $ sufficiently well, as this is precisely the training purpose of $ \mathcal{M}_{\widehat{\Theta}} $.
The training performance is evaluated in Subsection \ref{S52}.
\end{remark}%

\subsection{Convergence analysis}\label{A2}

The next theorem shows that solutions of \eqref{32} converge toward an $ x^* $-$ \mathcal{S} $-minimum-norm solution of the inverse problem \eqref{31}.

\begin{theorem}[convergence]\label{thm53}
Let $ y^\delta\in \mathcal{Y} $ with $ \left\|y - y^\delta\right\|_\mathcal{Y}\leq \delta $ and let $ \alpha(\delta) $ be such that $ \alpha(\delta)\to 0 $ and $ \delta^2/\alpha(\delta) \to 0 $ as $ \delta\to 0 $.
Then every sequence $ \{x_{\alpha_k}^{\delta_k}\} $, where $ \delta_k\to 0 $, $ \alpha_k:= \alpha(\delta_k) $ and $ x_{\alpha_k}^{\delta_k} $ is a solution of \eqref{32}, has a convergent subsequence. The limit of every convergent subsequence is an $ x^* $-$\mathcal{S}$-minimum-norm solution. 
If in addition, the $ x^* $-$\mathcal{S}$-minimum-norm solution $ x^\dagger $ is unique, then
\begin{displaymath}
\lim_{\delta\to 0}x^\delta_{\alpha(\delta)} = x^\dagger.
\end{displaymath}
\end{theorem}

\begin{proof}
Let $ \alpha $ and $ x_{\alpha_k}^{\delta_k} $ be as above and let $ x^\dagger $ be an $ x^* $-$\mathcal{S}$-minimum-norm solution. 
Then by the definition of $ x_{\alpha_k}^{\delta_k} $
\begin{displaymath}
\begin{aligned}
\left\|F(x_{\alpha_k}^{\delta_k}) - y^{\delta_k}\right\|^2_\mathcal{Y} + \left\|x^{\delta_k}_{\alpha_k} - x^*\right\|_{\mathcal{S}, \alpha}^2
=
&\left\|F(x_{\alpha_k}^{\delta_k}) - y^{\delta_k}\right\|^2_\mathcal{Y} + \alpha_k\left\|x^{\delta_k}_{\alpha_k} - x^*\right\|_{\mathcal{S}^+, 2}^2 + \left\|x^{\delta_k}_{\alpha_k}\right\|_{\mathcal{S}^-, 2}^2\\
\leq 
&\left\|F(x^\dagger) - y^{\delta_k}\right\|^2_\mathcal{Y} + \alpha_k\left\|x^\dagger - x^*\right\|_{\mathcal{S}^+, 2}^2 + \left\|x^\dagger\right\|_{\mathcal{S}^-, 2}^2\\
\leq
&\delta_k^2 + \alpha_k\left\|x^\dagger - x^*\right\|_{\mathcal{S}^+, 2}^2.
\end{aligned}
\end{displaymath}
and hence
\begin{equation}\label{311}
\lim_{k\to \infty}F(x_{\alpha_k}^{\delta_k}) = y,\qquad \lim_{k\to \infty}\left\|x^{\delta_k}_{\alpha_k}\right\|_{\mathcal{S}^-, 2} = 0
\end{equation}
and
\begin{equation}\label{312}
\underset{k\to\infty}{\lim\sup}\left\|x_{\alpha_k}^{\delta_k} - x^*\right\|_{\mathcal{S}^+, 2} \leq \left\|x^\dagger - x^*\right\|_{\mathcal{S}^+, 2}.
\end{equation}
Therefore, $ \{x_{\alpha_k}^{\delta_k}\} $ is bounded. Hence, there exist an element $ x\in\mathcal{X} $ and a subsequence again denoted by $ \{x_{\alpha_k}^{\delta_k}\} $ such that
\begin{equation}\label{313}
x_{\alpha_k}^{\delta_k}\rightharpoonup x\quad \mathrm{as}\quad k\to \infty.
\end{equation}
This together with \eqref{311} implies that $ x\in\mathcal{D}(F) $, $ \mathrm{supp}(x)\subset\mathcal{S} $ and that $ F(x) = y $. Now by the weak lower semicontinuity of the norm, \eqref{312}, \eqref{313} and the definition of $ x^\dagger $,
\begin{equation*}
\left\|x - x^*\right\|_{\mathcal{S}^+, 2} \leq \underset{k\to\infty}{\lim\sup}\left\|x_{\alpha_k}^{\delta_k} - x^*\right\|_{\mathcal{S}^+, 2} \leq \left\|x^\dagger - x^*\right\|_{\mathcal{S}^+, 2} \leq \left\|x - x^*\right\|_{\mathcal{S}^+, 2}.
\end{equation*}
This together with $ \mathrm{supp}(x)\subset\mathcal{S} $ implies that $ \left\|x - x^*\right\|_2 = \left\|x^\dagger - x^*\right\|_2 $, i.e., $ x $ is also an $ x^* $-$\mathcal{S}$-minimum-norm solution. From \eqref{312}, \eqref{313} and the fact that
\begin{displaymath}
\left\|x_{\alpha_k}^{\delta_k} - x\right\|_2^2 = \left\|x_{\alpha_k}^{\delta_k} - x^*\right\|_2^2 + \left\|x^* - x\right\|_2^2 + 2\langle x_{\alpha_k}^{\delta_k} - x^*, x^* - x\rangle
\end{displaymath}
yields that
\begin{displaymath}
\underset{k\to\infty}{\lim\sup}\left\|x_{\alpha_k}^{\delta_k} - x\right\|_2^2 \leq 2\left\|x^* - x\right\|_2^2 + 2\lim_{k\to \infty}\langle x_{\alpha_k}^{\delta_k} - x^*, x^* - x\rangle = 0
\end{displaymath}
and hence that
\begin{displaymath}
\lim_{k\to \infty}x_{\alpha_k}^{\delta_k}\to x.
\end{displaymath}
If $ x^\dagger $ is unique, the assertion about the convergence of $ x^\delta_{\alpha(\delta)} $ follows from the fact that every sequence has a subsequence converging towards $ x^\dagger $.
\end{proof}

We now turn to the convergence rate analysis. In the next theorem, we will give sufficient conditions for the rate $ \left\|x_\alpha^\delta - x^\dagger\right\|_{\mathcal{S}^+, 2} = O(\sqrt{\delta}) $ and $ \left\|x_\alpha^\delta - x^\dagger\right\|_{\mathcal{S}^-, 2} = O(\delta) $. 
These conditions will also imply the rate $ \left\|F(x_\alpha^\delta) - y\right\|_\mathcal{Y} = O(\delta) $ for the residual.

\begin{theorem}[convergence rate]\label{thm54}
Let $\mathcal{D}(F)$ be convex, let $ y^\delta\in \mathcal{Y} $ with $ \left\|y-y^\delta\right\|_\mathcal{Y}\leq\delta $ and let $ x^\dagger $ be an $ x^* $-$\mathcal{S}$-minimum-norm solution. Moreover, let the following conditions hold:
\begin{enumerate}[(i)]
\item F is Fr\'{e}chet-differentiable,
\item there exist $ \gamma\geq 0 $ such that $ \left\|F'(x^\dag) - F'(x)\right\|_\mathcal{Y} \leq \gamma\left\|x^\dagger - x\right\|_2 $ for all $ x\in\mathcal{D}(F) $ in a sufficiently large ball around $ x^\dagger $,
\item there exists $ \omega\in \mathcal{Y} $ satisfying $ x^\dagger - x^* = F'(x^\dagger)^*\omega $ and
\item  $\gamma\left\|\omega\right\|_\mathcal{Y}<1$.
\end{enumerate}
Then for the choice $ \alpha\sim\delta $, we obtain
\begin{displaymath}
\left\|x_\alpha^\delta - x^\dagger\right\|_{\mathcal{S}^+, 2} = O(\sqrt{\delta}),\quad \left\|x_\alpha^\delta - x^\dagger\right\|_{\mathcal{S}^-, 2} = O(\delta)\quad \mathrm{and} \quad \left\|F(x_\alpha^\delta) - y\right\|_\mathcal{Y} = O(\delta)
\end{displaymath}
\end{theorem}
\begin{proof}
Since $ x^\delta_\alpha $ is a minimizer of \eqref{32}, $ F(x^\dagger) = y $ and $ \left\|y-y^\delta\right\|\leq\delta $ imply that
\begin{displaymath}
\left\|F(x_{\alpha}^{\delta}) - y^\delta\right\|^2_\mathcal{Y} + \alpha\left\|x_{\alpha}^{\delta} - x^*\right\|_{\mathcal{S}^+, 2}^2 + \left\|x_{\alpha}^{\delta} - x^*\right\|_{\mathcal{S}^-, 2}^2
\leq 
\delta^2 + \alpha\left\|x^\dagger - x^*\right\|_{\mathcal{S}^+, 2}^2
\end{displaymath}
and hence that
\begin{equation}\label{314}
\begin{aligned}
&\left\|F(x_{\alpha}^{\delta}) - y^\delta\right\|^2_\mathcal{Y} + \alpha\left\|x_{\alpha}^{\delta} - x^\dagger\right\|_{\mathcal{S}^+, 2}^2 + \left\|x_{\alpha}^{\delta} - x^\dagger\right\|_{\mathcal{S}^-, 2}^2\\
\leq 
&\delta^2 + \alpha\left(\left\|x^\dagger - x^*\right\|_{\mathcal{S}^+, 2}^2 + \left\|x_{\alpha}^{\delta} - x^\dagger\right\|_{\mathcal{S}^+, 2}^2 - \left\|x_{\alpha}^{\delta} - x^*\right\|_{\mathcal{S}^+, 2}^2\right)\\
=
&\delta^2 + 2\alpha\left\langle x^\dagger - x^*, x^\dagger - x_{\alpha}^{\delta}\right\rangle.
\end{aligned}
\end{equation}
Let now $ \alpha\sim\delta $. Then \eqref{314} implies that $ x^\delta_\alpha\in B_\rho(x^\dagger) := \{x: \|x-x^\dagger\|_2<\rho\} $ for any fixed $ \rho > 2\left\|x^\dagger - x^*\right\|_2 $, provided $\delta$ is sufficiently small, which we assume to hold in the following. Therefore, condition \textit{(ii)} is applicable assuming that it holds for all $ x\in \mathcal{D}(F)\cap B_\rho(x^\dagger) $. Note that conditions \textit{(i)} and \textit{(ii)} imply that
\begin{equation}\label{315}
F(x^\delta_\alpha) = F(x^\dagger) + F'(x^\dagger)(x^\delta_\alpha - x^\dagger) + r^\delta_\alpha
\end{equation}
holds with
\begin{equation}\label{316}
\left\|r^\delta_\alpha\right\|_\mathcal{Y} \leq \dfrac{\gamma}{2}\left\|x^\delta_\alpha - x^\dagger\right\|_2^2.
\end{equation}
By condition (iii), \eqref{314} implies that
\begin{equation}\label{317}
\left\|F(x_{\alpha}^{\delta}) - y^\delta\right\|^2_\mathcal{Y} + \alpha\left\|x_{\alpha}^{\delta} - x^\dagger\right\|_2^2 + (1-\alpha)\left\|x_{\alpha}^{\delta} - x^\dagger\right\|_{\mathcal{S}^-, 2}^2\leq \delta^2 + 2\alpha\langle \omega, F'(x^\dagger)(x^\dagger - x_{\alpha}^{\delta})\rangle.
\end{equation}
Combining \eqref{315}--\eqref{317} leads to
\begin{displaymath}
\begin{aligned}
&\left\|F(x_{\alpha}^{\delta}) - y^\delta\right\|^2_\mathcal{Y} + \alpha\left\|x_{\alpha}^{\delta} - x^\dagger\right\|_2^2 + (1-\alpha)\left\|x_{\alpha}^{\delta} - x^\dagger\right\|_{\mathcal{S}^-, 2}^2\\
\leq 
&\delta^2 + 2\alpha\langle \omega, (y - y^\delta) + (y^\delta - F(x^\delta_\alpha)) + r^\delta_\alpha\rangle\\
\leq
&\delta^2 + 2\alpha\delta\left\|\omega\right\|_\mathcal{Y} + 2\alpha\left\|\omega\right\|_\mathcal{Y}\left\|F(x^\delta_\alpha) - y^\delta\right\| + \alpha\gamma\left\|\omega\right\|_\mathcal{Y}\left\|x^\delta_\alpha - x^\dagger\right\|^2_2
\end{aligned}
\end{displaymath}
and hence to
\begin{displaymath}
\left(\left\|F(x^\delta_\alpha) - y^\delta\right\|_\mathcal{Y} - \alpha\left\|\omega\right\|_\mathcal{Y}\right)^2 + \alpha(1 - \gamma\left\|\omega\right\|_\mathcal{Y})\left\|x^\delta_\alpha - x^\dagger\right\|^2_2 + (1-\alpha)\left\|x_{\alpha}^{\delta} - x^\dagger\right\|_{\mathcal{S}^-, 2}^2\leq (\delta + \alpha\left\|\omega\right\|_\mathcal{Y})^2
\end{displaymath}
This together with (iv) implies that
\begin{displaymath}
\left\|F(x^\delta_\alpha) - y^\delta\right\|_\mathcal{Y} \leq \delta + 2\alpha\left\|\omega\right\|_\mathcal{Y},\qquad \left\|x_{\alpha}^{\delta} - x^\dagger\right\|_{\mathcal{S}^-, 2} \leq \dfrac{\delta + \alpha\left\|\omega\right\|_\mathcal{Y}}{\sqrt{1-\alpha}}
\end{displaymath}
and that
\begin{displaymath}
\left\|x_{\alpha}^{\delta} - x^\dagger\right\|_{\mathcal{S}^+, 2} \leq \left\|x_{\alpha}^{\delta} - x^\dagger\right\|_2\leq \dfrac{\delta + \alpha\left\|\omega\right\|_\mathcal{Y}}{\sqrt{\alpha}\left(1 - \gamma\left\|\omega\right\|_\mathcal{Y}\right)^{\frac{1}{2}}}
\end{displaymath}
The assertion now follows with $ \alpha\sim\delta $.
\end{proof}

\begin{remark}\rm\label{R4}
As discussed before, the variational regularization method \eqref{32} has a similar structure with classical Tikhonov regularization, which under standard assumptions achieves $ \left\|x_\alpha^\delta - x^\dagger\right\|_2 = O(\sqrt{\delta}) $ \cite{EHW96}.
The above theorem shows that our method achieves a convergence rate of $ \left\|x_\alpha^\delta - x^\dagger\right\|_{\mathcal{S}^+, 2} = O(\sqrt{\delta}) $ and $ \left\|x_\alpha^\delta - x^\dagger\right\|_{\mathcal{S}^-, 2} = O(\delta) $, which is faster than the classical Tikhonov regularization (see Subsection \ref{S54} for numerical validation).
Such a theoretical result verifies the benefit introduced by $ \mathcal{S}\supset\mathrm{supp}(x^\dagger) $.
Moreover, we note that the convergence rate improves as the \textit{a priori} given $\mathcal{S}$ better approximates $ \mathrm{supp}(x^\dagger) $ (i.e., $ \wit{\bds{S}} $ approximates $ \bds{S} $ in Section \ref{S4}), which is exactly the training purpose of our deep neural network $\mathcal{M}_{\wih{\Theta}}$ (see Subsection \ref{S322}).
\end{remark}

The following proposition indicates that a solution of the variational problem \eqref{32} should be small outside $\mathcal{S}$ in the sense of $ L^2 $-norm as $ \alpha\to 0 $ and $ y^\delta\to y $.

\begin{prop}\label{P1}
Let $ y^\delta\in \mathcal{Y} $ with $ \left\|y - y^\delta\right\|_\mathcal{Y}\leq \delta $, $ \alpha>0 $ and $ x^\dagger $ be the exact solution of \eqref{31}, a solution $ x $ of \eqref{32} will satisfy
\begin{displaymath}
\int_{\Omega\setminus\mathcal{S}}|x(\lambda)|^2d\lambda \leq \delta^2 + \alpha\left\|x^\dagger - x^*\right\|_{\mathcal{S}^+, 2}^2.
\end{displaymath}
\end{prop}
\begin{proof}
This directly follows from the inequality
\begin{displaymath}
\begin{aligned}
&\left\|F(x) - y^\delta\right\|^2 + \left\|x - x^*\right\|_{\mathcal{S}, \alpha}^2\\
=
&\left\|F(x) - y^\delta\right\|^2 + \alpha\int_{\mathcal{S}}|x(\lambda) - x^*(\lambda)|^2d\lambda + \int_{\Omega\setminus\mathcal{S}}|x(\lambda) - x^*(\lambda)|^2d\lambda\\
=
&\left\|F(x) - y^\delta\right\|^2 + \alpha\int_{\mathcal{S}}|x(\lambda) - x^*(\lambda)|^2d\lambda + \int_{\Omega\setminus\mathcal{S}}|x(\lambda)|^2d\lambda\\
\leq 
&\left\|F(x^\dagger) - y^\delta\right\|^2 + \alpha\int_{\mathcal{S}}|x^\dagger(\lambda) - x^*(\lambda)|^2d\lambda + \int_{\Omega\setminus\mathcal{S}}|x^\dagger(\lambda)|^2d\lambda\\
\leq
&\delta^2 + \alpha\int_{\mathcal{S}}|x^\dagger(\lambda) - x^*(\lambda)|^2d\lambda.
\end{aligned}
\end{displaymath}
\end{proof}

\end{document}